\documentclass[a4paper,11pt]{article}
\usepackage[a4paper,left=26.5mm,right=26mm,top=26.25mm,bottom=36.75mm]{geometry}

\usepackage[english]{babel}	
\usepackage{lmodern}
\usepackage[T1]{fontenc}
\usepackage[utf8]{inputenc}	

\usepackage{amsfonts, amsmath, amsthm, amssymb} 
\usepackage{mathtools}
\usepackage{mathrsfs}	
\usepackage{bbm}	
\DeclareMathAlphabet{\mathbbold}{U}{bbold}{m}{n}	

\usepackage[leftbars]{changebar}
\usepackage{comment}
\usepackage{hyperref}

\usepackage{authblk} 

\newcommand{\R}{\mathbb{R}}
\newcommand{\N}{\mathbb{N}}

\renewcommand{\mid}{\;:\;}

\renewcommand{\le}{\leqslant}
\renewcommand{\ge}{\geqslant}
\renewcommand{\leq}{\leqslant}
\renewcommand{\geq}{\geqslant}

\renewcommand{\emptyset}{\varnothing}
\renewcommand{\subsetneq}{\varsubsetneq}

\newcommand{\eps}{\varepsilon}

\newcommand{\lang}{\left\langle}
\newcommand{\rang}{\right\rangle}

\newcommand{\T}[1]{{1}^\mathrm{T}}

\newcommand{\ie}{\emph{i.e.~}}
\newcommand{\eg}{\emph{e.g.~}}

\newcommand{\parfrac}[2]{\frac{\partial{#1}}{\partial{#2}}}

\newcommand{\Id}{\mathbbm{I}}

\def\startmodif{\begingroup} 
\def\stopmodif{\endgroup}

\DeclareMathOperator{\codim}{codim}
\DeclareMathOperator{\Image}{Im}
\DeclareMathOperator{\rank}{rank}

\newtheorem{theorem}{Theorem}[section]
\newtheorem{corollary}[theorem]{Corollary}
\newtheorem{lemma}[theorem]{Lemma}
\newtheorem{proposition}[theorem]{Proposition}

\theoremstyle{definition}

\newtheorem{remark}[theorem]{Remark}

\newtheorem*{issue}{Main issue}

\usepackage{enumitem}
\usepackage{array}
\usepackage{color}
\usepackage{stmaryrd}
\usepackage{dsfont}
\usepackage{pdflscape}
\usepackage{datetime}
\usepackage{cases}

\newcommand{\e}{\mathrm e}
\renewcommand{\S}{\mathbb S}
\newcommand*\diff{\mathop{}\!\mathrm{d}}

\newcommand{\xhat}{\hat{x}}
\renewcommand{\epsilon}{\varepsilon}
\renewcommand{\phi}{\varphi}
\newcommand{\End}{\mathrm{End}}

\newcommand{\GL}{\mathrm{GL}}
\newcommand{\OO}{\mathcal{O}}
\newcommand{\LL}{\mathcal{L}}

\newcommand{\Au}[1]{A_{#1}} 
\newcommand{\A}{{C^\infty(\R^n, \R)}}

\newcommand{\OOO}{\mathscr{O}}

\newcommand{\Gt}{\tilde\OOO} 
\newcommand{\deltat}{{\tilde\delta}}
\newcommand{\D}{D}

\newcommand{\K}{\mathcal{K}}
\newcommand{\Pric}{\xi}

\newcommand{\Kx}{\K_1}
\newcommand{\Keps}{\K_2}
\newcommand{\KP}{\K_3}

\newcommand{\GG}{G}
\newcommand{\FF}{F}
\newcommand{\HH}{H}

\newcommand{\NNN}{\mathscr{N}}

\newcommand{\fonction}[5]{
\begin{equation*}
\displaystyle
\begin{array}{lrcl}
{#1}: & #2 & \longrightarrow & #3 \\
    & #4 & \longmapsto & #5
\end{array}
\end{equation*}}

\newcommand{\feed}{\lambda}
\newcommand{\cont}{u}

\newcommand{\intset}[2]{\{#1, \dots, #2\}}
\newcommand{\norm}[1]{\left\Vert #1\right\Vert}

\newcommand{\abs}[1]{\left\vert #1\right\vert}
\newcommand{\I}{\mathcal{I}}
\newcommand{\km}{{k_{\I}}}
\newcommand{\di}{{|\I|}}

\newcommand{\kdep}{{j_0}}
\newcommand{\V}{V}

\newcommand{\C}{\mathcal{C}}

\newcommand{\Cx}{\C_1}
\newcommand{\Ceps}{\C_2}
\newcommand{\CP}{\C_3}

\newcommand{\Prics}{\mathcal{S}}
\newcommand{\sym}{\Prics}

\newcommand{\NFOT}{$\mathrm{(NFOT)}$}
\newcommand{\FC}{$\mathrm{(FC)}$}

\newcommand{\VV}{\mathcal{V}}

\newcommand{\Pan}{{\Pric}_{\omega}}
\newcommand{\xhatan}{{\xhat}_{\omega}}
\newcommand{\epsan}{{\eps}_{\omega}}

\newcommand{\LLL}{L}

\newcommand{\dif}{\mathrm{d}}
\newcommand{\supl}{u_M}
\newcommand{\RR}{r}
\newcommand{\RRR}{R}
\newcommand{\dmax}{\eta_1}

\newcommand{\UUU}{U}

\newcommand{\lambdasat}{\lambda_{\textrm{sat}}}

\makeatletter
\newcommand{\setword}[2]{%
  #1\def\@currentlabel{\unexpanded{#1}}\label{#2}%
}
\makeatother

\newcommand{\FFF}{\Lambda}

\ifpdf
\hypersetup{
  pdftitle={Avoiding observability singularities in output feedback bilinear systems},
  pdfauthor={L. Brivadis, J.-P. Gauthier, L. Sacchelli, and U. Serres}
}
\fi

\title{Avoiding observability singularities in output feedback\linebreak bilinear systems}

\date{\today}

\author[1]{Lucas Brivadis}
\author[2]{Jean-Paul Gauthier}
\author[1]{Ludovic Sacchelli}
\author[1]{Ulysse Serres}

\affil[1]{Univ. Lyon, Universit\'e Claude Bernard Lyon 1, CNRS, LAGEPP UMR 5007, 43 bd du 11 novembre 1918, F-69100 Villeurbanne, France}
\affil[2]{Universit\'e de Toulon, Aix Marseille Univ, CNRS, LIS, France
}

\begin{document}

\maketitle
\begin{abstract}
Control-affine output systems generically present observability singularities, \ie inputs that make the system unobservable.
This proves to be a difficulty in the context of output feedback stabilization, 
where this issue is usually discarded by uniform observability assumptions for state feedback stabilizable systems.
Focusing on state feedback stabilizable bilinear control systems with linear output,
we use a transversality approach to provide perturbations of the stabilizing state feedback law, in order to make our system observable in any time even in the presence of singular inputs.
\end{abstract}


\noindent {\bf Keywords:}
\begin{minipage}[t]{.8\linewidth}
\flushleft
Observability, Transversality theory, Output feedback, Stabilization
\end{minipage}


\section{Introduction}

Stabilizing the state of a dynamical system to a target point is a classical problem in control theory.
However, in many physical problems, only part of the state is known. Hence a state feedback can not be directly implemented.
When a stabilizing state feedback exists, a commonly used idea is to apply this feedback to an estimation of the state, relying on a dynamical system called the \emph{observer}, which learns the state of the system from its dynamics and the measured output.
This strategy belongs to the family of \emph{dynamic output feedback stabilization} techniques.

In the deterministic setting,
output feedback stabilization has been extensively studied (see \emph{e.g.} \cite{ AndrieuPraly2009,AtassiKhalil1999, Coron1994,EsfandiariKhalil1992,GauthierKupka1992, EsfandiariKhalil1993,   MarconiPralyIsidori2007,  TeelPraly1994, TeelPraly1995}).
The \emph{observability} of a controlled system for some fixed input qualifies the ability to estimate the state using its output, and characterizes the fact that two trajectories of the system can be  distinguished by their respective outputs over a given time interval. This crucial notion constitutes a field of study in itself (see \emph{e.g.} \cite{AndrieuPraly2009, Bernard-etal.2017, Gauthier_book, TW2009}).
A commonly used hypothesis to achieve output feedback stabilization is the \emph{uniform} observability of the system, that is the system is observable for all possible inputs.
It is well-known that a globally state feedback stabilizable system that is uniformly observable is also semi-globally output feedback stabilizable (see \emph{e.g.} \cite{EsfandiariKhalil1992, EsfandiariKhalil1993, TeelPraly1994, TeelPraly1995}).

However, as shown in \cite{Gauthier_book}, it is not generic for a dynamical system to be uniformly observable.
There may exist singular inputs for the system, that are inputs that make the system unobservable on any time interval, and the output feedback may produce such singular inputs.
This defeats the purpose of output feedback stabilization, which is still an open problem when such inputs exist.
Investigating this issue, some authors propose a different approach by allowing time-varying (either periodic as in \cite{Coron1994} or ``sample and hold'' as in \cite{ShimTeel2002}) output feedback. Doing so, the authors use a separation principle to show output feedback stabilization.
Adopting another point of view and in line with \cite{MarcAurele}, we are interested in smooth time-invariant output feedback.

In this work, we restrict ourselves to the class of
\startmodif
single-input single-output
\stopmodif
bilinear systems with linear observation that are state feedback stabilizable at some target point, which, with no loss of generality, is chosen to be $0$.
We also assume the system to be observable at the target, that is, the constant input obtained by evaluation of the feedback at 0 is not singular.
This class of systems is a natural choice of study for two reasons.
First, the uniform observability hypothesis is still not generic in this case. In particular, one can easily check that there generically exists constant inputs that make the system unobservable in any time.
Secondly, according to \cite{fliess}, any control-affine system with finite dimensional observation space may be immersed in such a system.

In this context, a natural question to ask is:
``Can we ensure that only observable inputs are produced by the dynamics when the output feedback is obtained as a combination of an observer and a stabilizing state feedback?''
This question falls within the more general and unsolved problem of building a smooth separation principle for systems with observability singularities.
\startmodif
One cannot hope for generic bilinear systems that all stabilizing state feedback laws ensure the observability of the closed-loop system. However, we show that for any stabilizing state feedback law, there exist small additive perturbations to this feedback that satisfy this observability property and conserve its locally stabilizing property.
Transversality theory is used
to prove the existence of such an open and dense class of perturbations.
In particular, for almost all considered systems,  almost any locally stabilizing feedback law ensures observability of the closed-loop system.
Stabilization by output feedback is beyond the scope of this paper, which focuses only on the observability issue. Yet, the obtained results may pave the way to the construction of a ``generic'' separation principle.
\stopmodif
For our results to hold,
some properties of the dynamical observer are needed.
The problem is tackled with a general observer design, and it is shown in a closing section that the classical Luenberger and Kalman observers fit our hypotheses.

\subsubsection*{Organization of the paper}

In Section~\ref{sec:state}, we state the main results of this paper. We begin this section with some definitions and notations, and we emphasise the precise issue. In particular, we define the system and the class of feedback perturbations we are interested in.
We then state our main results on observability properties of the perturbed system, and assert that the classical Kalman and Luenberger observers fit our hypotheses.

In Section~\ref{sec:proof} the reader may find a proof of our main results 
in three subsections. We rely on a transversality approach, which requires some technical preliminary results (Section~\ref{sec:prel}).
Sections \ref{Subsec:part2} and \ref{sec:target} are then  focused on the proof of our first main theorem and its corollary, respectively.

Lastly, we prove in Section~\ref{sec:appli}   that the Luenberger and Kalman observers fit our hypotheses, so that we can apply our previous theorems to these observers. In order to do so, we prove that their dynamics are somehow compatible with the Kalman observability decomposition.

\subsubsection*{Notations}

Let $\N$ be the set of non-negative integers.
For any subset $\I\subset\N$, $|\I|$ denotes its cardinality.

Let $n, m$ be positive integers.
Let $\lang\cdot,\cdot\rang$ be the canonical scalar product on $\R^n$, $|\cdot|$ the induced Euclidean norm,
$B(x,r)$ the open ball centered at $x$ of radius $r$ for this norm,
and $\S^{n-1} \subset \R^n$ the unit sphere.
Let $\LL(\R^n, \R^m)$ be the set of \startmodif linear maps \stopmodif from $\R^n$ to $\R^m$ and $\End(\R^n) = \LL(\R^n, \R^n)$.
For any endomorphism $A\in\End(\R^n)$, denote by $A^*$ its adjoint operator.

If $f$ is a function from $\R^{n}$ to $\R^{m}$,
the notation $D f(x)[v]$ stands for the differential at $x\in\R^{n}$ applied to the vector $v\in\R^{n}$ of the function $f$. 
The partial differential of $f$ at $x$ with respect to the variable $y$ is denoted by
$D_{y}f(x)$.
In particular,
for any function $t \mapsto v(t)$ defined on a real interval containing zero,
we use the shorthand notation $v^{(i)} = \frac{\diff^i v}{\diff t^i}(0)$ for all $i \in \N$.

Let $k\in\N$.
The set of all $k$-jets from $\R^{n}$ to $\R^{m}$ is denoted 
by $J^k(\R^{n}, \R^{m})$  (see, for instance, \cite[Chapter II]{GG}).
The mapping
\(
\sigma : J^k_x(\R^{n}, \R^{m}) \to \R^{n}
\)
given by
\(
\sigma : j^k f \mapsto \sigma \big(j^k f\big)=x
\)
is called the source map and the mapping
\(
\tau : J^k_x(\R^{n}, \R^{m}) \to \R^{m}
\)
given by
\(
\tau : j^k f \mapsto \tau \big(j^k f\big)=f(x)
\)
is called the target map.
Put $J^k_x(\R^{n}, \R^{m}) = \sigma^{-1}(x)$, $J^k(\R^{n}, \R^{m})_y = \tau^{-1}(y)$ and $J^k_x(\R^{n}, \R^{m})_y = \sigma^{-1}(x) \bigcap \tau^{-1}(y)$.
We have
\(
J^k(\R^{n}, \R^{m})
= \coprod_{x\in\R^{n}} J^k_x(\R^{n}, \R^{m}) = \R^{n} \times J^k_x(\R^{n}, \R^{m}).
\)

\section{Statement of the results}\label{sec:state}

\subsection{Problem statement}

Let $n$ be a positive integer, 
$A,B\in \End(\R^n)$, $C\in \LL(\R^n,\R)$, $b\in \R^n$ and $\cont\in C^\infty(\R_+, \R)$.
\startmodif
Set $\Au{u} = A +u B$.
In the present article, we focus on the following observed bilinear control system:
\stopmodif
\begin{equation}\label{E:observation_system}
\left\{
\begin{aligned}
&\dot{x}=\Au{u} x+ b\cont
\\
&y= C x.
\end{aligned}
\right.
\end{equation}

System~\eqref{E:observation_system} is said to be \emph{observable} in time $T>0$ and for the control function $\cont$ if and only if, for all pair of solutions $\big((x_1, y_1), (x_2, y_2)\big)$ of 
\eqref{E:observation_system},
$(y_1-y_2)|_{[0, T]} \equiv 0$ implies
$(x_1-x_2)|_{[0, T]} \equiv 0$.
\startmodif
For bilinear control systems of the form \eqref{E:observation_system}, we have the following characterization.
\begin{proposition}\label{prop:obs}
System~\eqref{E:observation_system} is observable in time $T$
for the control $u$ if and only if
for every $\omega_0 \in \S^{n-1}$ the solution of
$\dot \omega = \Au{u(t)}\omega$ initiated from $\omega_0$ satisfies $C\omega|_{[0, T]} \not\equiv 0$.

\end{proposition}

\stopmodif

If \eqref{E:observation_system} is observable for $u=0$ in some time $T>0$, then it is also observable in any time $T>0$, and we say that the pair $(C, A)$ is \emph{observable}.
According to the Kalman rank condition, $(C, A)$ is observable if and only if the rank of the following observability matrix
\begin{align}\label{E:Kalman-matrix}
\OO(C, A) = 
\begin{pmatrix}
C\\
CA\\
\vdots\\
CA^{n-1}
\end{pmatrix}
\end{align}
is equal to $n$.

Let $\Prics$ be a finite dimensional manifold and let $\LLL:\Prics\to \LL(\R, \R^n)$. For all $u\in \R$, let $f(\cdot,u)$ be a vector field over $\Prics$.
Denoting $\varepsilon= \xhat-x$, we introduce a dynamical observer system depending on the pair $(f, \LLL)$:
\begin{equation}\label{E:Observer_system}
\left\{
\begin{aligned}
\dot{\xhat}
&= \Au{u} \xhat +  b u - \LLL(\Pric)C\varepsilon
\\
\dot{\eps}
&=\left( \Au{u}  -\LLL(\Pric)C \right) \eps
\\
\dot{\Pric}
&= f(\Pric,u).
\end{aligned}
\right.
\end{equation}
Let $\feed\in C^\infty(\R^n, \R)$ be such that $0$ is an asymptotically stable equilibrium point of
the vector field $x \mapsto \Au{\lambda(x)}x + b\lambda(x)$ 
for some open domain of attraction $\D(\lambda)$.
We will further assume that $\lambda(0)=0$, which is true up to a substitution of $A$ with $A+\lambda(0) B$.

As stated in the introduction, our goal is to make system~\eqref{E:observation_system} observable 
in time $T$ for the control $u = \feed \circ \xhat$,
where $\xhat$ follows \eqref{E:Observer_system} with initial conditions $(\xhat_0, \eps_0, \Pric_0)$.
Since the stabilizing feedback $\lambda$ does not guarantee this property, we consider a small perturbation $\lambda + \delta$ of it.
For all $\delta\in\A$, we consider the coupled system
\begin{equation}\label{E:kalman_coupled}
\left\{
\begin{aligned}
\dot{\xhat}
&= \Au{(\feed+\delta)(\xhat)} \xhat + b (\feed+\delta)(\xhat) - \LLL(\Pric)C \eps
\\
\dot{\eps}
&= \left(\Au{(\feed+\delta)(\xhat)} - \LLL(\Pric) C \right) \eps
\\
\dot{\Pric}
&=f(\Pric,(\lambda+\delta)(\xhat))\\
\dot{\omega}
&= \Au{(\feed+\delta)(\xhat)} \omega.
\end{aligned}
\right.
\end{equation}

\startmodif
\begin{remark}
In system \eqref{E:kalman_coupled}, the dynamics of $(\xhat, \eps, \Pric)$ do not depend on $\omega$.
However, the dynamics of $\omega$ are included in \eqref{E:kalman_coupled} as they are crucial for the observability analysis of \eqref{E:observation_system} with input $u=\lambda(\xhat)$,
as stated in Proposition~\ref{prop:obs}.
We will sometimes consider $(\xhat, \eps, \Pric)$ to be the first coordinates of a solution of \eqref{E:kalman_coupled} without fixing any initial condition for $\omega$.
\end{remark}
\stopmodif

\startmodif
From now on, we denote by $\K=\Kx\times \Keps \times \KP$  a semi-algebraic compact subset of $\D(\lambda)\times \R^n\times\sym$, which stands for a subset of the space of initial conditions of system~\eqref{E:Observer_system}.
For all $R>0$, let
\[
\VV_R = \left\{ \delta \in \A \mid \forall x\in B(0, R),\quad \delta(x) = 0 \right\}.
\]
\stopmodif
We ask the observer given by $(f, \LLL)$ to satisfy the following important properties:

\vspace{.5cm}
\noindent
\begin{minipage}[t]{.1\linewidth}
\setword{\FC}{FC}
\end{minipage}
\begin{minipage}[t]{.85\linewidth}
(Forward completeness.)
For all $u\in C^{\infty}(\R_+,\R)$, the time-varying vector field $f(\cdot ,u)$ is forward complete.
Moreover, for all $(\xhat_0,\varepsilon_0,\Pric_0,\omega_0)\in \K \times \S^{n-1}$
and for all $\delta\in\A$ bounded over $\D(\lambda)$, the coupled system~\eqref{E:kalman_coupled} has a unique solution $(\xhat, \eps, \Pric, \omega)\in C^\infty(\R_+, \R^n\times\R^n\times\sym\times\S^{n-1})$ defined on $[0, +\infty)$.
\end{minipage}

\vspace{.5cm}
\noindent
\begin{minipage}[t]{.1\linewidth}
\setword{\NFOT}{NFOT}
\end{minipage}
\begin{minipage}[t]{.85\linewidth}
(No flat observer trajectories.)
For all $R>0$, there exists $\eta>0$ such that for all $\delta\in\VV_R$ satisfying $\sup\{\abs{\delta(x)}\mid x\in \Kx\} < \eta$,
for all $(\xhat_0,\varepsilon_0,\Pric_0,\omega_0)\in \K\times\S^{n-1}$ such that $(\xhat_0,\varepsilon_0)\neq (0,0)$,
there exists a positive integer $k$ such that the solution of~\eqref{E:kalman_coupled} with initial condition $(\xhat_0,\varepsilon_0,\Pric_0,\omega_0)$ satisfies $\xhat^{(k)}(0)\neq 0$.
\end{minipage}

\bigskip\noindent
These properties are investigated in the last section of the paper. There, we show that the classical Luenberger and Kalman observers fit these hypotheses so that the main results may be applied to these observers.

For all $k\in\N$, $K\subset\R^n$ and $\delta \in \A$, let
\begin{align*}
\norm{\delta}_{k, K}
&= \sup \left\{\abs{ \frac{\partial^\ell \delta}{\partial x_{i_1} \cdots \partial x_{i_\ell}}(x)} \mid 0\leq\ell\leq k,\quad 1\leq i_1\leq \cdots\leq i_\ell \leq n,\quad x\in K\right\}.
\end{align*}
For any $k \in \N$, any compact subset $K \subset \R^n$ and any $\eta>0$, $k \in \N$,
let
\[
\NNN(k, K, \eta)
= \left\{ \delta \in \A \mid \norm{\delta}_{k, K} < \eta \right\}.
\]


\startmodif
\begin{remark}
One can check that for any open subset $\UUU \subset \D(\lambda)$ relatively compact in $\D(\lambda)$,
for all $\RRR>0$,
there exists $\eta>0$ such that for all $\delta\in\VV_\RRR$ satisfying $|\delta|<\eta$, the feedback $\lambda+\delta$ is such that 0 is asymptotically stable with domain of attraction containing $\UUU$.
Hence in the following we focus only on the observability properties of the stabilizing feedback $\lambda+\delta$.
\end{remark}
\stopmodif


\begin{issue}
Let $T>0$.
Under genericity assumptions on $(A, B, C)$, does there exist $\RRR, \eta>0$, a positive integer $k$ and a residual set $\OOO\subset \NNN(k, \Kx, \eta)$ such that  we have the following property.  For all $\delta\in\OOO\cap\VV_\RRR$ and for all initial conditions $(\xhat_0, \eps_0, \Pric_0)\in\K$, system~\eqref{E:observation_system} is observable in time $T$ for the control $u = (\feed + \delta) \circ \xhat$, where $\xhat$ follows \eqref{E:kalman_coupled} with initial conditions $(\xhat_0, \eps_0, \Pric_0)$ and feedback perturbation $\delta$?
\end{issue}

\subsection{Main results}\label{sec:main_results}
In this section, we state the main results of the paper whose proofs are postponed to the upcoming sections.
We first state our main theorem, that deals with the observability of  system~\eqref{E:kalman_coupled}. Its proof is the most technical part of the paper, and heavily relies on transversality theory. 
\begin{theorem}\label{Thm:main}
Assume that the pairs $(C, A)$ and $(C, B)$ are observable.
Assume that $0\notin \Kx$.
Then there exist $\eta>0$, a positive integer $k$ and a dense open (in the Whitney $C^\infty$ topology) subset $\OOO \subset \NNN(k, \Kx, \eta)$ such that
the solution to \eqref{E:kalman_coupled} with $\delta\in\OOO$ and initial condition $(\hat x(0), \eps(0), \Pric(0), \omega(0))\in\K\times\S^{n-1}$ satisfies
\begin{equation}\label{main_eq}
\exists k_0 \in \{ 0,\dots,k\}\quad\mid\quad \left.\frac{\diff^{k_0}}{\diff t^{k_0}}\right|_{t=0} C\omega(t) \neq 0.   
\end{equation}

\end{theorem}
The proof of this theorem can be found in Section~\ref{Subsec:part2}.
\begin{remark}\label{rk:obs1}
Property \eqref{main_eq} is stronger than observability of \eqref{E:kalman_coupled} in any time $T>0$. This implication is shown in Corollary~\ref{obs}.
Pay attention to the assumption $0\notin \Kx$. In Section~\ref{sec:target}, 
this assumption is removed, while only slightly weakening our observability result. 
\end{remark}

Theorem~\ref{Thm:main} leads to the following corollary which states that under genericity assumptions on the system, there exists a generic class of perturbations $\delta$ such that the feedback $\lambda+\delta$ makes \eqref{E:kalman_coupled} observable.


\begin{corollary}\label{Cor:main}
Assume that the pairs $(C, A)$ and $(C, B)$ are observable.
Assume that $0$ is in the interior of $\Kx$.
Let $T>0$.
Then there exist $\RRR, \eta>0$, a positive integer $k$ and a dense open subset \startmodif
$\OOO \subset \NNN(k, \Kx, \eta) \cap \VV_\RRR$
such that
the solution to \eqref{E:kalman_coupled} with $\delta\in\OOO$
\stopmodif
and initial condition $(\xhat_0, \eps_0, \Pric_0, \omega_0)\in\K\times\S^{n-1}$ satisfies
\[
\exists t\in[0, T] \quad\mid\quad C \omega(t) \neq 0,
\]
that is system~\eqref{E:observation_system} is observable in time $T$ for the control $u = (\feed + \delta) \circ \xhat$,
where $\xhat$ follows \eqref{E:kalman_coupled} with initial conditions $(\xhat_0, \eps_0, \Pric_0)$ and feedback perturbation $\delta$.
\end{corollary}

\startmodif
This result also implies a generic observability property directly on the stabilizing state feedback law $\lambda$.
\begin{corollary}\label{cor:feedback}
Assume that the pairs $(C, A)$ and $(C, B)$ are observable.
Assume that $0$ is in the interior of $\Kx$.
Denote by $\Lambda$ the set of feedbacks $\lambda\in C^\infty(\R^n, \R)$ such that $0$ is a locally asymptotically stable equilibrium point of
the vector field $x \mapsto \Au{\lambda(x)}x + b\lambda(x)$.
Let $T>0$ and $\FFF_T\subset\Lambda$ be the set of feedbacks $\lambda\in\Lambda$ such that
\eqref{E:observation_system} is observable in time $T$ for the control $u = \feed \circ \xhat$, where $\xhat$ follows \eqref{E:kalman_coupled} with $\delta\equiv0$ and initial conditions $(\xhat_0, \eps_0, \Pric_0)$ in $\K$.
Then $\FFF_T$ is a dense open subset of $\Lambda$.
\end{corollary}
The proof of these two corollaries can be found in Section~\ref{sec:target}.
\stopmodif

\begin{remark}
\startmodif
Because $\VV_\RRR$ is not open in the Whitney $C^\infty$ topology,
the set $\OOO$ defined in
Corollary~\ref{Cor:main}
is not open in the Whitney $C^\infty$ topology,
but it is open in the induced topology on $\NNN(k, \Kx, \eta) \cap \VV_\RRR$.
\stopmodif
Also, the set of matrices $(A, B, C)\in \End(\R^n)^2\times\LL(\R^n, \R)$ such that $(C, A)$ and $(C, B)$ are both observable is open and dense.
As a consequence,
``$(C, A)$ and $(C, B)$ are observable'' is a generic hypothesis.
\startmodif
Contrarily to the strategy followed in \cite{MarcAurele} on some specific example, the results of this paper do not explicitly design any perturbation $\delta\in\OOO$, but rather state that for almost all bilinear system, almost all perturbation $\delta\in\NNN(k, \Kx, \eta) \cap\VV_\RRR$ belongs to $\OOO$ (in a topological sense).
\stopmodif
\end{remark}

Finally,
the next theorem shows that the classical Luenberger and Kalman observers fit hypotheses~\ref{FC} and \ref{NFOT}. Hence, our results may be applied to these well-known observers.

\begin{theorem}\label{thm:appli}
Assume that $(C, A)$ is observable.
Assume that $\lambda$ is bounded over $\D(\lambda)$.
Let $Q\in\sym_n$.
For all $\Pric\in\sym_n$ and all $u\in\R$,
consider the following well-known observers:
\begin{align}
&f^\mathrm{Luenberger}(\Pric, u) = 0\tag{Luenberger observer}\\
&f^\mathrm{Kalman}_{Q}(\Pric, u) = \Pric \Au{u}^* + \Au{u} \Pric + Q - \Pric C^*C\Pric \tag{Kalman observer}
\end{align}
and $\LLL(\xi) = \Pric C^*$.
Then the coupled system~\eqref{E:kalman_coupled} given by $(f, \LLL)$ satisfies the hypotheses~\ref{FC} and \ref{NFOT} for any $f\in\{f^\mathrm{Luenberger}, f^\mathrm{Kalman}_{Q}\}$.
\end{theorem}
The proof of this theorem can be found in Section~\ref{sec:appli}.

\begin{remark}
If $\lambda$ is unbounded over $\D(\lambda)$, then for any open subset $\UUU$ relatively compact in $\D(\lambda)$,
we can obtain by smooth saturation of $\lambda$ a new bounded feedback law $\lambdasat$ such that ${\lambdasat}_{|U}={\lambda}_{|U}$, for which the previous statement holds. (In particular $U\subset\D(\lambdasat)$.)
\end{remark}

\section{Proofs of the observability statements}\label{sec:proof}

In order to prove our main Theorem~\ref{Thm:main} and its Corollary~\ref{Cor:main}, we need a series of preliminary results that we state and prove below.
The main results will appear as corollaries of these subsequent lemmas.

Before we start the more technical elements of the paper, let us present the method we follow in order to prove the main results.
Theorem~\ref{Thm:main} is an application of transversality theory to our particular problem
(see \cite{GMP} for the statements we rely on; see also \cite{abraham1967transversal,GG}).
Consider a solution to \eqref{E:kalman_coupled} for a given perturbation $\delta$ of the feedback law, and a set of initial conditions in $\K\times \S^{n-1}$. We set  $h:C^{\infty}(\R^n,\R)\times (\K\times \S^{n-1})\times \R^+\to \R$ to be the smooth map given by
$$
h(\delta,(\xhat_0, \eps_0, \Pric_0,\omega_0),t)=C\omega(t).
$$
As stated in Section~\ref{sec:state}, to get  observability after perturbation of the feedback,  we would like to show that there exists $\delta$, preferably small, such that 
\begin{equation}\label{E:obser_f}
t\mapsto h(\delta,z_0,t)\neq0,\qquad \forall z_0=(\xhat_0, \eps_0, \Pric_0,\omega_0)\in \K\times \S^{n-1}.
\end{equation}
\startmodif
A sufficient condition for $\delta$ to satisfy \eqref{E:obser_f} is that for each $z_0\in \K\times \S^{n-1}$, there exists  an integer $k$ such that
\(
\left.\frac{\diff^k }{\diff t^k}\right|_{t=0}(h(\delta ,z_0,t))\neq 0.
\)
\stopmodif
In other words,
our goal will be achieved if we can prove that there exists $\delta$ and a finite set $\I\subset \N$  such that the map $H:C^{\infty}(\R^n,\R)\times (\K\times \S^{n-1})\to \R^{|\I|}$ given by
\[
H(\delta,z_0)=\left(\left.\frac{\diff^k }{\diff t^k}\right|_{t=0}h(\delta,z_0,t)\right)_{k\in \I},
\]
never vanishes.
This is where transversality theory comes into play. Let $N$ denote the dimension of the surrounding space of $\K\times \S^{n-1}$.  
We can ensure that there exists $\delta$ satisfying \eqref{E:obser_f} if we can prove that for some choice of $\I$, with $|\I|>N$, $H$ is \emph{transversal} to $\{0\}$ at $\delta =0$.
That is to say, if we can prove that the rank of the map $H(0,\cdot)$ is maximal, equal to $|\I|>N$, at any of its vanishing points (at which point $H(0,\cdot)$ is then a submersion).

Now it should be noted that in general, proving that a map is transversal to a point is a major hurdle, especially if the dimensions $n$ and $N$ of the spaces are unspecified. As a general rule, 
considering more orders of derivation of $h$ greatly increases the degrees of freedom of the map $H$ (by including higher order derivatives of $v$, as jet spaces grow exponentially in dimension), while only slightly increasing the size of the target space. This points towards an augmentation of the rank of $H$, making a proof of transversality achievable.

The difficulty lies however in producing a ``rank increasing property'' on $H$ as $|\I|$ increases. That is, finding a symmetry in the successive derivatives of $h$ that proves that for any dimension, a  set $\I$ can be found by differentiating $h$ sufficiently many times.

The symmetry we use to prove the rank condition on the map $H$ can be described as follows.
For $k\in \N$, let 
$$
h^k(\delta,z_0,t)=CB^k\omega(t).
$$
It turns out that if $h^{k+1}(0,z_0,\cdot)$ has a non-zero derivative of any order (including order 0), then we automatically get the rank condition for $h^{k}(0,z_0,\cdot)$
(this statement will be made precise in Corollary \ref{C:rank_N}). 

Here the hypothesis that  $(C,B)$ is an observable pair becomes crucial. 
Indeed, observe that $h^{k}(0,z_0,0)=CB^k\omega_0$. Hence, for any $\omega_0\in \S^{n-1}$ there exists a $k\in\{0,\dots,n-1\}$ such that 
$$
h^{k}(0, z_0, 0)\neq 0.
$$
This in turns induces a partition of $\K\times \S^{n-1}$ into $n$ subsets on each of which at least one of the maps $h^0,\dots ,h^{n-1}$ never vanishes. Since 
$h^{k+1}(0,z_0,\cdot)$ not vanishing implies that the rank condition is satisfied for $h^{k}(0,z_0,\cdot)$, we chain-apply  $n$ successive transversality theorems to prove the existence of a $\delta$ such that $h(\delta,z_0,\cdot)$ has always at least one non-zero time derivative at any point   $z_0\in \K\times \S^{n-1}$.

Section~\ref{sec:prel} is aimed at making explicit the connection between the rank condition and the family of maps $(h^k)_{k\in\N}$. Section~\ref{Subsec:part2} is dedicated to the effective application of the principles presented in this introduction, which leads to the proof of Theorem~\ref{Thm:main}. Section~\ref{sec:target} concludes the proof of the observability statements by taking into account the behavior of the system near the target $0$.

\subsection{Preliminary results}\label{sec:prel}

Let $u\in C^\infty(\R_+,\R)$ and  consider the ordinary differential equation
\begin{equation}\label{E:omega}
\dot{\omega}=\left(A+u(t)B\right) \omega.
\end{equation}
For all $k,m\in \N$, let $\FF_k^m: C^\infty(\R_+,\R) \times \R^n \to \R$ be the function such that 
$$
\FF_k^m(u, \omega_0)=CB^m\omega^{(k)}(0)
$$
where $t\mapsto \omega(t) $ is the solution of \eqref{E:omega} with initial condition $\omega_0$.

Let us introduce the $n \times n$ matrix valued polynomials in the indeterminates $X_0, \dots, X_{k-1}$ by:
\begin{equation*}
\End(\R^n)[X_0,\dots X_{k-1}]
= \begin{cases}
\End(\R^n) &\mbox{if } k=0 \\
\End(\R^n)[X_0,\dots X_{k-2}][X_{k-1}] &\mbox{otherwise},
\end{cases}
\end{equation*}
and set
\[
\End(\R^n)\left[ (X_k)_{k \in \N} \right] = \bigcup_{k \in \N}\End(\R^n)[X_0,\dots X_{k-1}].
\]
\startmodif
Let $\Psi : \End(\R^n)\left[ (X_k)_{k \in \N} \right] \to \End(\R^n)\left[ (X_k)_{k \in \N} \right]$ be the linear map defined by
\[
\Psi (P)(X_0,\dots ,X_{k}) = P(X_0,\dots, X_{k-1})(A+X_0B)+\sum_{i=0}^{k-1}\frac{\partial P}{ \partial X_i}\left(X_0,\dots, X_{k-1}\right) X_{i+1},
\]
where $k = \min \left\{ \ell \in \N \mid P \in \End(\R^n)[X_0,\dots X_{\ell-1}] \right\}$.
\stopmodif

Finally,
let us define the family $(P_k)_{k\in \N}$ of matrix valued polynomials such that
$P_0 \in \End(\R^n)$ and $P_k\in  \End(\R^n)[X_0,\dots X_{k-1}]$, for all $k \ge 1$,
by
\begin{equation}\label{Eq:recurrence_relation}
P_0=\Id, \qquad
P_{k+1}
= \Psi (P_{k}), \qquad \forall k\in \N.
\end{equation}
It is clear\footnote{
Note that, for $k \neq 0$, the function $\FF_k^m$ actually acts on $(k-1)$-jets at zero of functions and not on functions themselves.
Consequently,
the restriction
\(
\left. \FF_k^m \right|_{J^\ell_0(\R, \R) \times \R^n}
\)
is well-defined as soon as $\ell \ge k-1$.
Of course, for $k=0$, the restriction
\(
\left. \FF_0^m \right|_{J^\ell_0(\R, \R) \times \R^n}
\)
makes sense only if $\ell \ge 0$.
In summary,
the restriction
\(
\left. \FF_k^m \right|_{J^\ell_0(\R, \R) \times \R^n}
\)
is well-defined as soon as $\ell \ge k$.}
that for all $m\in \N$,
\[
\FF_k^m(u, \omega_0)
= \begin{cases}
CB^m\omega_0	&\mbox{if~} k=0 \\
CB^mP_k\left(u^{(0)},u^{(1)},\dots , u^{(k-1)}\right)\omega_0	&\mbox{otherwise},
\end{cases}
\]
where $u^{(i)}$ is shorthand for $\frac{\diff^i u}{\diff t^i}(0)$ for all $i\in \N$.
For all $k\in \N$ and $i\in \N$, $1\leq i \leq k$, let  $Q_i^k=\dfrac{\partial P_k}{\partial X_{k-i}}$.

\begin{lemma}\label{lemma:polynomial}
For all $i\in \N\setminus\{0\}$, there exist $R_i^0,\dots, R_i^{i-1}\in  \End(\R^n)[X_0,\dots X_{i-1}]$ such that 
\footnote{Actually, we can show that $R_i^0,\dots, R_i^{i-1}\in  \End(\R^n)[X_0,\dots X_{i-2}]$}
\[
Q_{i}^{i+k} = \sum_{j=0}^{i-1} k^j R_i^j, \qquad \forall k \ge 0.
\]
Furthermore,
$\displaystyle R_{i}^{i-1}=\frac{BP_{i-1}}{(i-1)!}$.
\end{lemma}

\begin{proof}
We prove the first part of the statement by induction on $i$.
For $i=1$, one easily checks that %
\begin{equation}\label{Eq:initialisation:i=1}
Q_1^{1+k} = B, \qquad \forall k \in \N.
\end{equation}
Assuming the desired property for $i$, we have to prove that
there exist $R_{i+1}^0, \dots, R_{i+1}^{i} \in \End(\R^n)[X_0,\dots X_{i}]$ such that
\[
Q_{i+1}^{i+1+k} = \sum_{j=0}^{i} k^j R_{i+1}^{j}, \qquad \forall k \ge 0.
\]
Using the definition of $Q_{i+1}^{i+1 +\ell}$ and the recurrence relation \eqref{Eq:recurrence_relation}
yields
\begin{equation}\label{Eq:increments}
Q_{i+1}^{i+1 +\ell}
= \Psi(Q_{i}^{i+\ell}) + Q_{i+1}^{i+\ell}, \qquad \forall \ell \ge 1.
\end{equation}
Consequently,
for all $k \ge 0$,
\begin{align}
Q_{i+1}^{i+1+ k}
& = \sum_{\ell = 1}^{k} \left( Q_{i+1}^{i+1+\ell} - Q_{i+1}^{i+\ell} \right) + Q_{i+1}^{i+1}
\nonumber\\
& = \sum_{\ell = 1}^{k} \left(  \Psi(Q_{i}^{i+\ell})\right) + Q_{i+1}^{i+1}
\tag{by \eqref{Eq:increments}}\\
& = \sum_{\ell = 1}^{k} \left(  \sum_{j=0}^{i-1} \ell^j \Psi(R_{i}^{j})\right) + Q_{i+1}^{i+1}
\tag{by induction hypothesis}\\
& = \sum_{j=0}^{i-1}\left(  \sum_{\ell = 1}^{k} \ell^j \right) \Psi(R_{i}^{j}) + Q_{i+1}^{i+1}
\nonumber\\
& = \sum_{j=0}^{i-1} S^j(k) \Psi(R_{i}^{j}) + Q_{i+1}^{i+1}, \qquad \text{with~}S^j(k)=\sum_{\ell = 1}^{k} \ell^j.
\nonumber
\end{align}
Note that $Q_{i+1}^{i+1},
\Psi(R_{i}^{j}) \in \End(\R^n)[X_0, \dots, X_i]$ for all $j \in \{0, \dots, i-1\}$
($Q_{i+1}^{i+1} = \partial{P_{i+1}}/\partial{X_0}$).
Moreover, according to Faulhaber's formula, %
we have
\[
S^j(k) = \frac{k^{j+1}}{j+1}  + T^j(k), \qquad \forall j, k\in \N,
\]
where $T^j(k)$ is a polynomial in the variable $k$ of degree $j$ with no constant term.
Consequently,
\begin{align}
Q_{i+1}^{i+1+ k}
&= \frac{k^{i}}{i}\Psi(R_{i}^{i-1})  + \left( T^{i-1}(k)\Psi(R_{i}^{i-1}) + \sum_{j=0}^{i-2}S^j(k) \Psi(R_{i}^{j}) \right) + Q_{i+1}^{i+1}
\nonumber\\
&= k^{i} R_{i+1}^{i}  + \sum_{j=1}^{i-1} k^j R_{i+1}^j + R_{i+1}^0
\nonumber\\
&= \sum_{j=0}^{i} k^j R_{i+1}^j,
\nonumber 
\end{align}
with
$R_{i+1}^i = \Psi(R_{i}^{i-1})/i$, $R_{i+1}^0 = Q_{i+1}^{i+1}$
and
$R_{i+1}^{j} \in \End(\R^n)[X_0, \dots, X_i]$  for all $j \in \{0, \dots, i\}$.

The second part of the statement easily follows by induction.
Indeed,
\begin{align*}
BP_0 = Q_1^1 = \sum_{j=0}^{0} 0^j R_1^j = R_1^0,
\end{align*}
and
\[
R_{i+1}^i
= \frac{\Psi(R_{i}^{i-1})}{i}  
= \frac{1}{i} \Psi \left( \frac{1}{(i-1)!}BP_{i-1} \right) 
= \frac{1}{i!} B\Psi(P_{i-1})
= \frac{1}{i!} BP_i.
\]
The statement follows.
\end{proof}

\begin{corollary}\label{C:CBPk}
Let $i, m\in \N$, $i \ge 1$.
Let $v\in\R^i$ and $\omega_0\in \R^n$.
Either there exists $k_0\geq i$ such that $CB^m Q_i^k(v) \omega_0\neq 0$ for all $k\geq k_0$
or $C B^mQ_i^k(v)\omega_0=0$ for all $k\geq i$.
\end{corollary}
\begin{proof}
By Lemma \ref{lemma:polynomial}, we have
$Q_i^k=\sum_{j=0}^{i-1} (k-i)^j R_i^j$ for all integer $k\geq  i$. If $CB^m R_i^j(v) \omega_0= 0$ for all $j \in \{0, \dots, i-1\}$, then $C B^mQ_i^k(v)\omega_0=0$ for all $k\geq i$.
Otherwise, there exists $j \in \{0, \dots, i-1\}$ such that $CB^m R_i^j(v) \omega_0\neq 0$. Let $(k_0,\dots k_{i-1}) \in\N^i$ with $k_0<\dots<k_{i-1}$.
We have
\begin{align*}
C B^m
\begin{pmatrix}
Q_i^{i+k_0}(v)\\
\vdots\\
Q_i^{i+k_{i-1}}(v)
\end{pmatrix}
\omega_0
=
\begin{pmatrix}
1&k_0&\dots&k_0^i\\
\vdots&\vdots&&\vdots\\
1&k_{i-1}&\dots&k_{i-1}^i\\
\end{pmatrix}
C B^m
\begin{pmatrix}
R_i^{0}(v)\\
\vdots\\
R_i^{{i-1}}(v)
\end{pmatrix}
\omega_0.
\end{align*}
Since $k_0,\dots k_{i-1}$ are pairwise different, the Vandermonde matrix is invertible.
Consequently, there exits $j\in\{0, \dots, i-1\}$ such that $C B^mQ_i^{i+k_j}(v)\omega_0\neq0$. Hence, there exists at most $i-1$ positive integers $k_j$ such that $C B^mQ_i^{i+k_j}(v)\omega_0=0$. Thus, there exists $k_0 \geq i$ such that $CB^m Q_i^k(v) \omega_0\neq 0$ for all $k\geq k_0$.
\end{proof}

For all $P\in\End(\R^n)[X_0,\dots X_{k-1}]$ and all $v\in\R^\N$, we set $P(v) = P(v_0,\dots,v_{k-1})$.

\begin{corollary}\label{C:rank_N}
Let $v\in\R^\N$, $\omega_0\in \R^n$ and $m\in \N$.
If there exists $i \in \N\setminus\{0\}$ such that $CB^{m+1}P_{i-1}(v)\omega_0\neq 0$,
then there exists $k_0\in \N$ such that, for all $N\in \N \setminus \{0\}$,
\startmodif
the mapping
\footnote{Note that $\phi(\cdot)=\FF_{\{k_0,\dots, k_0+N-1\}}^m(\cdot,\omega_0)$, with $\FF_{\{k_0,\dots, k_0+N-1\}}^m$ defined as in Section~\ref{Subsec:part2}}
$\phi:J^{k_0+N-1}_0(\R, \R)=\R^{k_0+N}\to\R^N$ defined by
$$\phi(\cdot) = (CB^m P_{k_0}(\cdot)\omega_0,\dots,CB^mP_{k_0+N-1}(\cdot)\omega_0)$$
\stopmodif
has a rank $N$ differential at $(v_0,\dots,v_{k_0+N-1})$.
\end{corollary}
\begin{proof}
Assume that there exists $i \ge 1$ such that $CB^{m+1}P_{i-1}(v)\omega_0\neq 0$.
Since, according to Lemma \ref{lemma:polynomial},  $R_{i}^{i-1}=BP_{i-1}/(i-1)!$,
this is equivalent to $CB^{m} R_{i}^{i-1}(v) \omega_0 \neq 0$.
Thus, reasoning as in the proof of Corollary~\ref{C:CBPk},
the sequence $\big(CB^{m}Q_i^k(v)\omega_0\big)_{k\geq i}$ is not constant equal to zero.
Set
\begin{equation}\label{E:first-non-zero-diagonal}
i_0 = \min \left\{ i \in \N\setminus\{0\} : \big(CB^{m}Q_i^k(v)\omega_0\big)_{k\geq i} \not\equiv 0 \right\}.
\end{equation}
As a consequence of Corollary~\ref{C:CBPk}, there exists $k_0\in \N$  such that $C B^mQ_{i_0}^k(v) \omega_0\neq 0$ for all $k\geq k_0$, \ie
\[
\dfrac{\partial \left(CB^m P_{k}\omega_0\right)}{\partial X_{k-i_0}}(v_0,\dots,v_{k_0+N-1}) =
\dfrac{\partial \left(CB^m P_{k}\omega_0\right)}{\partial X_{k-i_0}}(v) \neq 0,
\qquad \forall k\geq k_0,
\]
and (by construction of $i_0$)
\[
\dfrac{\partial \left(CB^m P_{k}\omega_0\right)}{\partial X_{\ell}}(v_0,\dots,v_{k_0+N-1}) =
\dfrac{\partial \left(CB^m P_{k}\omega_0\right)}{\partial X_{\ell}}(v) = 0,   \qquad  \forall \ell>k-i_0.
\]
In other words,
\begin{equation}\label{E:differential-CB^mPk}
D\phi(v_0,\dots,v_{k_0+N-1}) =
\begin{pmatrix}
*		& \hdots	& *		& a_0(v)	& 0		&		& \hdots	&		& 0 \\
\vdots	&		&		& \ddots	& \ddots	& \ddots	&		&		& \vdots\\
*		& 		&  \hdots	&		& * 		& a_{N-1}(v) 	& 0		& \hdots	& 0\\
\end{pmatrix},
\end{equation}
with $a_i(v) = C B^mQ_{i_0}^{k_0+i}(v) \omega_0$.
The statement follows.
\end{proof}

\subsection{Observability away from the target and proof of Theorem~\ref{Thm:main}}
\label{Subsec:part2}

Using the results of the previous section, we are now able to prove our main Theorem~\ref{Thm:main}.
In this section, we assume that $0\notin\Kx$.
From now on $t \mapsto \left(\xhat(t), \eps(t), \Pric(t), \omega(t) \right)$, or simply $(\xhat, \eps, \Pric, \omega)$, denotes
the solution to \eqref{E:kalman_coupled} with initial condition $(\xhat_0, \eps_0, \Pric_0, \omega_0)$.

Let us introduce some new notation.
For any $k\in\N$, define the map $\GG^k$ by:
\fonction{\GG^k}{J^k(\R^n, \R) \times \Keps\times \KP}{J^k_0(\R, \R)}{ \Big( j^k\delta(\xhat_0), \epsilon_0, \Pric_0 \Big) }{j^k \big((\feed+\delta) \circ \xhat \big)(0).}
For any finite subset $\I\subset\N$ and any $m\in\N$,
set $\km = \max \I$ and
define the maps, $\FF_{\I}^m$ and $\HH^{m}_{\I}$ as follows:
\fonction{\FF_{\I}^m}{J^\km_0(\R, \R)\times \S^{n-1}}{\R^{\di}}{(v, \omega_0)}{ \big( CB^m P_{k}(v)\omega_0 \big)_{k\in\I},}
\[
\HH^{m}_{\I}
= \FF_{\I}^m \circ \left( \GG^{\km} \times \Id_{\S^{n-1}}\right).
\]
\begin{remark}\label{R:on-the-def-of-F}
Notice that for any $m, k_0 \in \N$ and any $N \in \N\setminus\{0\}$ such that $\I \subset \{k_0, \dots, k_0+N-1\}$, the map $\FF_{\I}^m$ satisfies
\[
\FF_{\I}^m
= \pi_{\I} \circ \FF_{\{k_0, \dots, k_0+N-1\}}^m,
\]
where  $\pi_{\I}: J^{k_0+N-1}_0(\R, \R) = \R^{k_0+N} \to \R^{|\I|}$ denotes the canonical projection onto the factors that correspond to indices in $\I$.
\end{remark}
Now we state the following proposition, which leads directly to Theorem~\ref{Thm:main}.
\begin{proposition}\label{P:main}
\startmodif
For all $m \in \{0, \dots, n-1\}$, define
\begin{equation*}
E_m
= \begin{cases}
\S^{n-1} &\mbox{if } m=0 \\
\left\{\omega_0\in \S^{n-1} \mid CB^i\omega_0 = 0,\quad ~\forall i \in\intset{0}{m-1}\right\} &\mbox{otherwise}.
\end{cases}
\end{equation*}
\stopmodif
Suppose $(C, A)$ and $(C, B)$ are observable pairs.
Then for every $m \in \intset{0}{n-1}$,
there exist $k\in\N$, a positive real number $\eta$ and a dense open subset $\OOO_m\subset \NNN(k, \Kx, \eta)$
such that for all $(\delta, \xhat_0, \epsilon_0, \Pric_0, \omega_0) \in \OOO_m \times \K\times E_m$
\[
\HH^m_{\intset{0}{k}}(j^{k}\delta(\xhat_0), \epsilon_0, \Pric_0, \omega_0) \neq 0.
\]
\end{proposition}
\begin{proof}
The proof strongly relies on the results of Section~\ref{sec:prel} and on the Goresky-MacPherson transversality theorem (see \cite[Part I, Chapter 1]{GMP}).
We prove the proposition by finite descending induction on $m$.
%
Note that since the pair $(C, B)$ is observable, we have
\(
\emptyset = E_n \subset E_{n-1} \subset \cdots \subset E_1 \subsetneq E_0 =  \S^{n-1}.
\)

\medskip
For $m=n-1$, the result is immediate because, by observability of the pair $(C, B)$, $CB^{n-1}\omega_0 \neq 0$ for all $\omega_0 \in E_{n-1}$.
Hence,
for $k=0$ and any positive real number $\eta$, we have for all $(\delta, \xhat_0, \epsilon_0, \Pric_0, \omega_0) \in \NNN(k, \Kx, \eta) \times \K\times E_{n-1}$,
\[
\HH^{n-1}_{\{0\}}(j^{0}\delta(\xhat_0), \epsilon_0, \Pric_0, \omega_0) = CB^{n-1}\omega_0 \neq 0.
\]

\medskip
Now suppose $1 \le m \le n-1$.
Note that, by definition of $E_{m-1} \setminus E_{m}$,
\begin{equation}\label{E:on_E_m-1-E_m}
CB^{m-1}\omega_0 \neq 0,
\quad \forall \omega_0 \in E_{m-1} \setminus E_{m}.
\end{equation}

Assume
that we are given a $k\in\N$, a positive real number $\eta$ and a dense open subset $\OOO_m \subset \NNN(k, \Kx, \eta)$ such that
\begin{equation}\label{E:induction-hypothesis}
\HH^m_{\{0, \dots, k\}}(j^{k}\delta(\xhat_0), \epsilon_0, \Pric_0, \omega_0) \neq 0,
\qquad \forall (\delta, \xhat_0, \epsilon_0, \Pric_0, \omega_0) \in \OOO_m \times \K\times E_m.
\end{equation}

Choose $(\delta, \xhat_0, \epsilon_0, \Pric_0, \omega_0) \in \OOO_m \times \K\times E_m$ and put $u(t) = (\lambda+\delta)\big( \xhat(t) \big)$.
Equation~\eqref{E:induction-hypothesis} implies that
$CB^m P_{i}(u^{(0)}, \dots, u^{(k)})\omega_0\neq0$ for an integer $i \in \{0, \dots, k\}$,
so,
by Corollary~\ref{C:rank_N}
there exists $k_0\in\N$ such that, for any positive integer $k_1$,
the map $\FF_{\{k_0, \dots, k_0+k_1-1\}}^{m-1}$ has a rank $k_1$ differential at $(u^{(0)}, \dots, u^{(k_0+k_1-1)})$.

Let $i_0 \in \N$ be defined as in the proof of Corollary~\ref{C:rank_N}.
Let $p\in\N\setminus\{0\}$ be such that $\xhat^{(p)}\neq 0$ and $\xhat^{(q)} = 0$ for all $q<p$ (which exists by hypothesis~\ref{NFOT} and $0\notin\Kx$),
and choose $\ell \in \{1, \dots n\}$ so that $\xhat_\ell^{(p)}\neq 0$.
Put
\[
\kdep
= \min\big\{ j\geq k_0 \mid j - i_0  \equiv 0 \pmod p \big\}
\footnote{Index $j_0$ corresponds to the smallest index $j \ge k_0$ such that $\xhat^{(p)}_\ell$ appears in $u^{(j-i_0)}$.}
\quad\text{and}\quad
\I = \big\{ \kdep + rp \mid r\in\intset{0}{N-1} \big\},
\]
where $N$ is a positive integer.
The (partial) differential of $\GG^{m}_{\I}$ with respect to
\[
w
= \left.\left( \delta, \parfrac{}{x_\ell}\delta, \dots, \left(\parfrac{}{x_\ell}\right)^{\km}\delta \right)\right|_{x=\xhat_0}
\]
at $X_0 = (j^{\km}\delta(\xhat_0), \epsilon_0, \Pric_0, \omega_0)$ is
the submatrix $D_w\GG^{m}_{\I}(X_0)$ obtained from $D\GG^{m}_{\I}(X_0)$ by deleting all columns that do not correspond to partial derivatives with respect to $w$.
In other words,
\[
D_w\GG^{m}_{\I}(X_0)
=\begin{pmatrix}
\mathtt{col}({0}) & \cdots & \mathtt{col}({\km-1})
\end{pmatrix}.
\]
Each column $\mathtt{col}({i})$, $i \in \{0, \dots, \km-1\}$ of $D_w\GG^{m}_{\I}(X_0)$ satisfies
\begin{align*}
\mathtt{col}({i})^* =
\begin{pmatrix}
0	&\cdots	&0	&b_i(X_0)	&*	&\cdots *
\end{pmatrix}^*,\qquad b_i(X_0) \neq 0,
\end{align*}
where the non zero coefficient $b_i(X_0)$ appears at the $ip\,$th row.
According to Fa{\`a} di Bruno formula, we have
\[
b_i(X_0) = n_i\left(\xhat^{(p)}_\ell \right)^i,
\]
$n_i$ being a positive integer for each $i \in \{0, \dots, \km-1\}$.

It is clear from the definition of $\FF_{\I}^m$ and Remark \ref{R:on-the-def-of-F} thereafter that $D\FF_{\I}^m$ is the submatrix of $D\FF_{\{k_0, \dots, \km\}}^m$
(see equation~\eqref{E:differential-CB^mPk})
obtained by keeping the $i\,$th rows for $i \in \I$.
Therefore,
\begin{align}
\rank \left( D\HH^{m}_{\I}(X_0) \right)
&\ge \rank \left( D_v\FF_{\I}^m\left(\GG^{\km}(X_0), \omega_0\right) \circ D_w \GG^{\km}(X_0) \right)
\nonumber\\
&= \rank \begin{pmatrix}
*		& \cdots	& *		& c_0(X_0)	& 0		&				& \cdots	&		& 0		\\
\vdots	&		&		& \ddots			& \ddots	& \ddots			&		&		& \vdots	\\
*		& 		&  \cdots	&				& * 		& c_{N-1}(X_0)	& 0		& \cdots	& 0		\\
\end{pmatrix},
\nonumber
\end{align}
where
\(
c_r(X_0) = a_{j_0+rp}\left(\GG^{\km}(X_0), \omega_0\right) b_{j_0+rp}(X_0),
\)
$r \in \{0, \dots, N-1\}$.
Hence $\HH_\I^{m-1}$ has a rank $N$ differential at $X_0$.

\bigskip
For any $k \in \N$, any compact subset $K \subset \R^n$ and any $\eta>0$, $k \in \N$,
define
\[
\mathcal{M}(k, K, \eta) = \left\{ \alpha \in J^k(\R^n, \R) : \exists f \in \NNN(k, K, \eta), \quad \exists a \in K, \quad \alpha = j^k f(a) \right\}.
\]
Clearly, $\mathcal{M}(k, K, \eta)$ is an open submanifold of $J^k(\R^n, \R)$.

Since the rank is a semi-continuous map, there exists a neighborhood
\(
\V \subset \mathcal{M}(\km, \Kx, \eta) \times \Keps \times \KP \times E_m
\)
of $(j^{\km}_0(\xhat_0), \epsilon_0, \Pric_0, \omega_0)$ such that $\HH_\I^{m-1}$ has a rank $N$ on $\V$.
Let $\rho \in (0, \eta)$ and $\C(\rho) = \Cx \times \Ceps \times \CP \times \Omega_m$ be a semi-algebraic compact subset of $\K\times E_m$ such that
\[
W := \mathcal{M}(\km, \Kx, \rho) \times \Ceps \times \CP \times \Omega_m \subset \V.
\]
Let $B = \big(\HH_\I^{m-1}|_W \big)^{-1}(0)$ and $Z = \pi(B)$, where $\pi$ is the projection that is parallel to $\Ceps \times \CP \times \Omega_m$.
Then,
and because $\Ceps \times \CP \times \Omega_m$ is compact, $Z\subset \mathcal{M}(\km, \Kx, \rho)$ is a closed semi-algebraic subset.
Hence,
according to the Goresky-McPherson transversality theorem (\cite[Part I, Chapter 1, page 38, Proposition]{GMP}), the set
\[
\Gt(\rho)
= \left\{ f \in C^\infty\big(\R^n, \mathcal{M}(\km, \Kx, \rho)\big) \mid
    f|_{\Cx} \text{~is transversal to~} Z \right\}
\]
is open and dense (in the Whitney $C^\infty$ topology) in $C^\infty\big(\R^n, \mathcal{M}(\km, \Kx, \rho)\big)$.
Moreover,
since $\HH_\I^{m-1}|_W$ is a submersion, we have
$\codim_{\mathcal{M}(\km, \Kx, \rho)} Z \geq \codim_{\R^N}\{0\} - \dim (\C(\rho)\times E_m) = N - \dim (\C(\rho)\times E_m)$. 
Picking $N$ sufficiently large, we have
\[
\codim_{\mathcal{M}(\km, \Kx, \rho)} Z > n
\]
in which case, transversal necessarily means to avoid.
It follows that
\begin{align*}
\Gt(\rho)
&= \left\{ f \in C^\infty\big(\R^n, \mathcal{M}(\km, \Kx, \rho)\big) \mid
    \forall \xhat\in\Cx, f(\xhat) \notin Z \right\}
\\
&= \left\{ f \in C^\infty\big(\R^n, \mathcal{M}(\km, \Kx, \rho)\big) \mid
    \forall (\xhat, \eps, \Pric, \omega)\in \C(\rho),
\big( f(\xhat), \eps, \Pric, \omega \big) \notin B \right\}
\\
&= \left\{ f \in C^\infty\big(\R^n, \mathcal{M}(\km, \Kx, \rho)\big) \mid
    \forall (\xhat, \eps, \Pric, \omega)\in \C(\rho),
\HH_\I^{m-1}\big( f(\xhat), \eps, \Pric, \omega \big) \neq 0 \right\}.
\end{align*}
By compactness of $\K\times E_m$,
there exists $q \in \N$ such that
\begin{equation}\label{E:compacite}
\K\times E_m
= \bigcup_{i=1}^{q} \C(\rho_i).
\end{equation}
Set
\(
\eta
= \min \{\rho_i \mid i=1, \dots, q\}>0,
\)
\(
k = \max \{\km(\rho_i) \mid i=1, \dots, q\}
\)
and define
\(
\Gt
= \bigcap_{i=1}^q \Gt(\rho_i).
\)
According to \eqref{E:compacite},
\begin{multline*}
\Gt
= \Big\{ f \in C^\infty \big(\R^n, \mathcal{M}(k, \Kx, \eta) \big) \mid 
\forall (\xhat, \eps, \Pric, \omega)\in \K\times E_m,\\
\HH_{\{0, \dots, k\}}^{m-1} \big(f(\xhat), \eps, \Pric, \omega \big) \neq 0 \Big\}.
\end{multline*}
Also,
by definition of $E_{m-1}$ and $E_m$,
$\HH_{\{0\}}^{m-1}(\omega) = CB^{m-1}\omega \neq 0$ for all $\omega\in E_{m-1} \setminus E_m$.
Thus,
\begin{multline*}
\Gt
= \Big\{ f \in C^\infty \big(\R^n, \mathcal{M}(k, \Kx, \eta) \big) \mid \forall (\xhat, \eps, \Pric, \omega)\in \K\times E_{m-1},\\
\HH_{\{0, \dots, k\}}^{m-1} \big(f(\xhat), \eps, \Pric, \omega \big) \neq 0 \Big\}
\end{multline*}
is an open dense subset of $C^\infty(\R^n, \mathcal{M}(k, \Kx, \eta))$.
Then $\OOO_{m-1} := \{
\tau\circ f\mid f\in \Gt\}$ where $\tau$ is the target map is an open dense subset
of $ \NNN(k, \Kx, \eta)$ and
\begin{multline*}
\OOO_{m-1}
= \Big\{ \delta \in \NNN(k, \Kx, \eta) \mid \forall (\xhat_0, \eps_0, \Pric_0, \omega_0)\in \K\times E_{m-1},\\
\HH_{\{0, \dots, k\}}^{m-1}(j^{k}\delta(\xhat_0), \epsilon_0, \Pric_0, \omega_0) \neq 0 \Big\}.
\end{multline*}
It concludes the induction and the proof.
\end{proof}

\begin{proof}[Proof of Theorem~\ref{Thm:main}]
Applying Proposition~\ref{P:main} to $m=0$ and recalling the definition of $H^0_{\{0,\dots,k\}}$, we immediately get the main Theorem~\ref{Thm:main}.
\end{proof}
A straightforward consequence of Theorem~\ref{Thm:main} is the following corollary, that deals with the observability of \eqref{E:observation_system}, as announced in Remark \ref{rk:obs1}.

\begin{corollary}\label{obs}
Assume that $(C, A)$ and $(C, B)$ are observable pairs.
\startmodif
Assume that $0\notin \Kx$.
\stopmodif
Then there exist $\eta>0$, $k\in\N$ and an open dense subset $\OOO \subset \NNN(k, \Kx, \eta)$ such that for all $(\delta, \xhat_0, \eps_0, \Pric_0)\in \OOO \times \K$, system~\eqref{E:observation_system} is observable in any time $T>0$ for the control $u = (\feed + \delta) \circ \xhat$,
where $\xhat$ follows \eqref{E:kalman_coupled} with initial conditions $(\xhat_0, \eps_0, \Pric_0)$ and feedback perturbation $\delta$.
\end{corollary}
\begin{proof}
Applying Proposition~\ref{P:main} to $m = 0$, we find that there exist $\eta>0$, $k\in\N$ and an open dense subset $\OOO \subset \NNN(k, \Kx, \eta)$ such that for all $(\delta, \xhat_0, \epsilon_0, \Pric_0, \omega_0) \in \OOO\times \K\times E_0$, $\HH^0_{\{0, \dots, k\}}(j^{k}\delta(\xhat_0), \epsilon_0, \Pric_0, \omega_0) \neq 0$.
Let $(\delta, \xhat_0, \epsilon_0, \Pric_0, \omega_0) \in \OOO\times \K \times \S^{n-1}$,
and let $(\xhat, \epsilon, \Pric, \omega)$ denote the solution of \eqref{E:kalman_coupled} with initial conditions $(\xhat_0, \epsilon_0, \Pric_0, \omega_0)$.
From the definition of $\HH^0_{\{0, \dots, k\}}$ it follows that there exists $i\in\N$ such that $C\omega^{(i)}(0) \neq 0$.
Consequently,
$C\omega|_{[0, T]} \not\equiv 0$, which was to be proved.
\end{proof}

As stated in Remark \ref{rk:obs1}, we now want to complete the compact $\Kx$ with a neighborhood of zero as in Corollary~\ref{Cor:main}.
We do so in the following section.

\subsection{Observability near the target and proof of Corollary~\ref{Cor:main}}
\label{sec:target}

We use Theorem~\ref{Thm:main} to prove Corollary~\ref{Cor:main}. In order to do so, we need the following notations and lemmas.
For any control $u\in C^\infty(\R_+,\R)$, let $\Phi_u:\R_+\to \End (\R^n)$ be the flow of the time-varying linear ordinary differential equation~\eqref{E:omega}.
So $\Phi_u(t)\omega_0$ is the solution of \eqref{E:omega} at time $t\in\R_+$ with initial condition $\omega_0\in\R^n$.
Notice for instance that $\Phi_0(t)=\e^{At}$.
Recall that
an input $u\in C^\infty(\R_+, \R)$ is said to make system~\eqref{E:observation_system} observable in time $T>0$ if for all $\omega_0\in \S^{n-1}$ there exists $t\in[0, T]$ such that $C\Phi_u(t)\omega_0\neq 0$.

\begin{lemma}\label{L:lemma1}
Let $T>0$, 
$\eta_0 = \max\{\left|
C\Phi_0(t)\omega_0
\right|\mid t\in [0,T],~\omega_0\in \S^{n-1} \}$
and $u\in C^\infty(\R_+, \R)$. If
\begin{equation}\label{E:hypo-lem}
\forall t\in [0,T],
\forall \omega_0\in \S^{n-1},\quad
\left|
C\Phi_u(t) \omega_0-C\Phi_0(t)\omega_0
\right|
<
\eta_0,
\end{equation}
then $u$ makes system~\eqref{E:observation_system} observable in time $T$.
\end{lemma}
\begin{proof}
Let $t\in[0, T]$ and $\omega_0\in \S^{n-1}$ be such that $|C\Phi_0(t)\omega_0| = \eta_0$.
Using \eqref{E:hypo-lem}, we get
\begin{align*}
\left|
C\Phi_u(t) \omega_0
\right|
&\geq
\left|
C\Phi_0(t) \omega_0
\right|
-
\left|
C\Phi_u(t) \omega_0-C\Phi_0(t) \omega_0
\right|
> 0,
\end{align*}
which shows that $u$ makes system~\eqref{E:observation_system} observable in time $T$.
\end{proof}

\begin{lemma}\label{L:lemma2}
Let $T>0$.
Let $M=\sup\{\|\Phi_0(t)\|\mid t\in [0,T]\}$.
Let $u\in C^\infty(\R_+, \R)$ and let $u_M=\sup\{|u(t)|\mid t\in [0,T]\}$. Then there exists a constant $K>0$ such that for all $t\in [0,T]$ and all $\omega_0\in \S^{n-1}$,
\begin{equation}\label{E:estimation}
\left|
\Phi_u(t)\omega_0 -\Phi_0(t)\omega_0
\right|
<
MKu_M\e^{K u_M}.
\end{equation}
\end{lemma}

\begin{proof}
By the variation of constants formula, for all $t\in [0,T]$ and all $\omega_0\in\S^{n-1}$,
$$
\Phi_u(t)\omega_0 -\Phi_0(t)\omega_0
=
\int_0^t\Phi_0(t-s)Bu(s)\Phi_u(s)\dif s
\,
\omega_0.
$$
Iterating integrals, we get a (formal) series expansion
\begin{equation}\label{E:interated_int}
\int_0^{s_0}\Phi_0(s_0-s_1)Bu(s)\Phi_u(s)\dif s_1
=
\sum_{k=0}^{+\infty} J_k
\end{equation}
where 
$$
J_k
=
\int_{0}^{s_0}
\dotsi
\int_{0}^{s_k}
\Psi_k(s_0,\dots,s_{k+1})
\Phi_0(s_{k+1})
u(s_0)
\dotsi
u(s_{k+1})
\dif s_1
\dotsm
\dif s_{k+1}
$$
with
$\Psi_k(s_0,\dots,s_{k+1})=\Phi_0(s_0-s_1)
B
\dotsm
\Phi_0(s_{k}-s_{k+1})
B$.

Then
$
\|\Psi_k(s_0,\dots,s_{k+1})\|\leq M^{k+1} \|B\|^{k+1}
$
and
$$
\|J_k\|\leq 
M^{k+2} \|B\|^{k+1} u_M^{k+1}\int_{0}^{s_0}
\dotsi
\int_{0}^{s_k}
\dif s_1
\dotsm
\dif s_{k+1}
\leq 
M^{k+2} \|B\|^{k+1} u_M^{k+1}
\frac{T^{k+1}}{(k+1)!}.
$$
Thus
$$
\begin{aligned}
\sum_{k=0}^{+\infty} \|J_k\|
&\leq 
\sum_{k=0}^{+\infty}
M^{k+2} \|B\|^{k+1} u_M^{k+1}
\frac{T^{k+1}}{(k+1)!}
\\
&\leq
M^{2} \|B\| u_M
T
\sum_{k=0}^{+\infty}
M^{k} \|B\|^{k} u_M^{k}
\frac{T^{k}}{k!}
\end{aligned}
$$
which proves the convergence of the series expansion~\eqref{E:interated_int} and inequality \eqref{E:estimation} with
\(
K = M\|B\|T.
\)
\end{proof}

\begin{proposition}\label{Prop_zero}
Assume that the pair $(C,A)$ is observable.
Assume that $0$ is in the interior of $\Kx$.
Let $T>0$.
Then there exists $\RRR>0$ such that $B(0,\RRR)\subset\Kx$ and $\dmax>0$ such that the following property holds:

Let $(\xhat,\varepsilon,\Pric,\omega)$ be the solution of \eqref{E:kalman_coupled}
with initial condition  $(\xhat_0,\varepsilon_0,\Pric_0,\omega_0)\in B(0,\RRR)\times \R^n\times\sym\times \S^{n-1}$.
Let $\delta\in\A$ such that $\delta(0)=0$ and $\sup\{|\delta(x)|\mid x\in\Kx\}<\dmax$.
If $\xhat(t)\in B(0,\RRR)$ for all $t\in [0,T]$, then the control $u:t\mapsto (\lambda+\delta)(\xhat(t))$ makes system~\eqref{E:observation_system} observable in time $T$.
\end{proposition}

\begin{proof}
Let $T>0$ and $\eta_0$ be as in the statement of Lemma~\ref{L:lemma1}.
The observability of the pair $(C,A)$ yields $\eta_0>0$.
Let $\dmax>0$ be such that $MK\dmax\e^{K \dmax} < \eta_0$.
For all $R>0$ and all $\delta\in\A$ satisfying $\delta(0)=0$ and $\sup\{|\delta(x)|\mid x\in\Kx\}<\dmax$, let $\supl(\RRR, \delta)=\sup\{|(\lambda+\delta)(x)|\mid x\in B(0,\RRR)\}$.
Since $\lambda+\delta$ is continuous and $\lambda(0)=\delta(0)=0$, $\supl(\cdot, \delta)$ is a continuous non decreasing function on $\R_+$ such that $\supl(0, 0)=0$ and $\supl(\RRR, \delta) \leq \supl(\RRR, 0) + \dmax$.
Then, we can choose $\RRR>0$ such that $MK(\supl(\RRR, 0) + \dmax)\e^{K (\supl(\RRR, 0) + \dmax)}< \eta_0$. 
Since $u_M(\cdot, 0)$ is non decreasing, it is possible to choose $\RRR$ such that $B(0,\RRR)\subset\Kx$.

Now, fix $\delta\in\A$ satisfying $\delta(0)=0$ and $\sup\{|\delta(x)|\mid x\in\Kx\}<\dmax$.
Let $(\xhat,\varepsilon,\Pric,\omega)$ be the solution of \eqref{E:kalman_coupled}
with initial condition  $(\xhat_0,\varepsilon_0,\Pric_0,\omega_0)\in B(0,\RRR)\times \R^n\times\sym\times \S^{n-1}$.
Then $MK\supl(\RRR, \delta)\e^{K \supl(\RRR, \delta)}< \eta_0$.
Hence, from Lemmas \ref{L:lemma1} and \ref{L:lemma2},
if $\xhat(t)\in B(0,\RRR)$ for all $t\in [0,T]$, then the control $u:t\mapsto (\lambda+\delta)(\xhat(t))$ makes system~\eqref{E:observation_system} observable in time $T$.
\end{proof}

\startmodif
\begin{proof}[Proof of Corollary~\ref{Cor:main}]
Let $\RRR>0$ and $\dmax$ be as in Proposition~\ref{Prop_zero}.
Let $\RR\in (0, \RRR)$ and $\rho\in(0, \RR)$.
We apply Corollary~\ref{obs} to the compact $\Kx \setminus B(0, \RR)$.
Since the statement holds for some $\eta$ small enough, we assume without loss of generality that $\eta<\dmax$:
there exist $\eta\in(0, \eta_1)$, $k\in\N$ and an open dense subset $\OOO \subset \NNN(k, \Kx\setminus B(0, \RR), \eta)$ such that for all $(\delta, \xhat_0, \eps_0, \Pric_0)\in \OOO \times \left(\Kx \setminus B(0, \RR)\right) \times \Keps \times \KP$, system~\eqref{E:observation_system} is observable in any time $T>0$ for the control $u = (\feed + \delta) \circ \xhat$,
where $\xhat$ follows \eqref{E:kalman_coupled} with initial conditions $(\xhat_0, \eps_0, \Pric_0)$ and feedback perturbation $\delta$.

Let
$$\OOO' = \left\{\deltat\in \NNN(k, \Kx, \eta)\cap \VV_\rho\mid \exists \delta\in\OOO, \forall x\in \Kx\setminus B(0, \RR), \deltat(x)=\delta(x)\right\}.$$
Then $\OOO'$ is open and dense in $\NNN(k, \Kx, \eta)\cap \VV_\rho$ (in the Whitney $C^\infty$ induced topology) since $\OOO$ is open and dense in $\NNN(k, \Kx\setminus B(0, \RR), \eta)$.
Moreover, if $\deltat\in\OOO'$, then
system~\eqref{E:observation_system} is still observable in any time $T>0$ for the control $u = (\feed + \deltat) \circ \xhat$ with initial conditions $(\xhat_0, \eps_0, \Pric_0)$ in $\left(\Kx \setminus B(0, \RR)\right) \times \Keps \times \KP$.

Let $(\deltat, \xhat_0, \eps_0, \Pric_0)\in \OOO' \times \K$.
If $\xhat_0\notin B(0, \RR)$, then the result holds from above.
On the other hand, assume that $\xhat_0\in B(0, \RR)$.
If $\xhat(t)\in B(0, \RRR)$ for all $t\in[0, T]$, then according to Proposition~\ref{Prop_zero}, \eqref{E:observation_system} is observable in time $T$ for the control $u = (\feed + \deltat) \circ \xhat$. Otherwise, there exists $t_0\in(0, T)$ such that $\xhat(t_0)\notin B(0, \RR)$. Apply Corollary~\ref{obs} with the new initial condition $(\xhat(t_0), \eps(t_0), \Pric(t_0))$ and with the same perturbation $\deltat$. Then \eqref{E:observation_system} is observable in time $T>t_0$ for the control $u = (\feed + \deltat) \circ \xhat$.
\end{proof}

\begin{proof}[Proof of Corollary~\ref{cor:feedback}]
Let $T>0$ and $\feed\in\Lambda$.
Let $\RRR$, $\eta$, $k$ and $\OOO$ be as in Corollary~\ref{Cor:main}.
Since $\OOO$ is dense (in the Whitney $C^\infty$ topology) in $\NNN(k, \Kx, \eta) \cap \VV_\RRR$,
for all neighborhood $\mathscr U$ of $\lambda\in \Lambda$, there exists $\delta \in \OOO$ such that $\feed+\delta\in \mathscr U\cap\FFF_T$.
Hence, $\FFF_T$ is a dense subset of $\Lambda$.
Moreover,
\begin{align*}
    \FFF_T
    &= \left\{\lambda\in\Lambda\mid
    \forall (\xhat_0, \eps_0, \Pric_0, \omega_0)\in\K\times\S^{n-1},
    \exists t\in[0, T], C\omega(t)\neq0
    \right\}\\
    &= \bigcap_{(\xhat_0, \eps_0, \Pric_0, \omega_0)\in\K\times\S^{n-1}}
    h_{\xhat_0, \eps_0, \Pric_0, \omega_0}^{-1}(\C^\infty([0, T], \R)\setminus\{0\})
\end{align*}
where $h_{\xhat_0, \eps_0, \Pric_0, \omega_0}:\Lambda\to C^\infty([0, T], \R)$
is given by
$h_{\xhat_0, \eps_0, \Pric_0, \omega_0}(\lambda) = C\omega|_{[0, T]}$
where $\omega$ is the solution of \eqref{E:kalman_coupled} with initial condition $(\xhat_0, \eps_0, \Pric_0, \omega_0)$ and $\delta\equiv0$.
The map $h$ is continuous,
the set $\C^\infty([0, T], \R)\setminus\{0\}$ is open and the set $\K\times\S^{n-1}$ is compact.
Thus $\FFF_T$ is open in $\Lambda$.
\end{proof}
\stopmodif

\section{Application to classical observers}
\label{sec:appli}

In this section, we show that there exist observers such that the key hypotheses~\ref{FC} and \ref{NFOT} are satisfied.
In particular, we show that both the Luenberger observer and the Kalman observer satisfy these hypotheses, as stated in Theorem~\ref{thm:appli}.
Hence, the main Theorem~\ref{Thm:main} and its Corollary~\ref{Cor:main} apply to these observers.
While \ref{FC} has already been studied for such observers
(see \eg \cite{Besancon, Gauthier_book}),
\ref{NFOT} is more difficult to check, and relies on the fact that the observer dynamics is somehow compatible with the Kalman observability decomposition.

For the sake of generality, 
we state the results of this section for an arbitrary output dimension $m$ (\ie $C\in\LL(\R^n,\R^m)$).
Let $\sym_n \subset \End(\R^n)$ denote the subset of real positive-definite symmetric endomorphism on $\R^n$.

Regarding hypothesis~\ref{FC}, the following result is well-known.

\begin{proposition}\label{prop:fc}
Assume that $\lambda$ is bounded over $\D(\lambda)$.
Let $Q\in\sym_n$.
For all $\Pric\in\sym_n$ and all $u\in\R$,
consider the following well-known observers:
\begin{align}
&f^\mathrm{Luenberger}(\Pric, u) = 0\tag{Luenberger observer}\\
&f^\mathrm{Kalman}_{Q}(\Pric, u) = \Pric \Au{u}^* + \Au{u}\Pric + Q - \Pric C^*C\Pric \tag{Kalman observer}
\end{align}
and $\LLL(\xi) = \Pric C^*$.
Then the coupled system~\eqref{E:kalman_coupled} given by $(f, \LLL)$ satisfies the hypothesis~\ref{FC} for any $f\in\{f^\mathrm{Luenberger}, f^\mathrm{Kalman}_{Q}\}$.
\end{proposition}

Let us investigate hypothesis~\ref{NFOT}.
First, we state sufficient conditions for it to hold, and then show that they are satisfied by both the Kalman and  Luenberger observers.

For all $A_0\in C^{\infty}\left(\R_+,\End(\R^n)\right)$ and for all $C_0 \in \LL(\R^n,\R^m)$, let $f(\cdot, A_0, C_0)$ be a forward complete time-varying vector field over $\sym_n$.
Let $\LLL:\Prics_n\times\LL(\R^n,\R^m) \to \LL(\R^m, \R^n)$.
For all $T\in\GL(\R^n)$, for all $(\bar{A}, \bar{C})\in\End(\R^n)\times\LL(\R^n,\R^m)$ and for all $\Pric\in \sym_n$, let $(\bar{f}, \bar{\LLL})$ be defined by
\begin{align}
\begin{cases}
    \bar{f}(T\Pric T^*, T\bar{A}T^{-1}, \bar{C}T^{-1}) = Tf(\Pric, \bar{A}, \bar{C})T^*\\
    \bar{\LLL}(T\Pric T^*, \bar{C}T^{-1}) = T\LLL(\Pric, \bar{C}).
\end{cases}\label{Hyp:T}
\end{align}
For all $(\bar{A}, \bar{C}, \bar{b}) \in \End(\R^n) \times \LL(\R^n,\R^m)\times\R^n$, we consider the following dynamical observer system
\begin{equation}\label{E:Observer_system_const}
\left\{
\begin{aligned}
&\dot{\xhat}=\bar{A} \xhat +  \bar{b} - \bar{\LLL}(\Pric, \bar{C})\bar{C}\varepsilon
\\
&\dot{\eps}=\left( \bar{A}  -\bar{\LLL}(\Pric, \bar{C})\bar{C} \right) \eps
\\
&\dot{\Pric}= \bar{f}(\Pric,\bar{A}, \bar{C}).
\end{aligned}
\right.
\end{equation}
For all $k\in\{1,\dots,n\}$, let $(\bar{A}, \bar{C}) \in \End(\R^n) \times \LL(\R^n,\R^m)$ having
the following structure:
\begin{equation}\label{eq:struct}
\begin{aligned}
\bar{A} = 
\begin{pmatrix}
A_{11} & 0\\
A_{21} & A_{22}
\end{pmatrix},\qquad
\bar{C} = 
\begin{pmatrix}
C_1 & 0
\end{pmatrix},
\end{aligned}
\end{equation}
with suitable matrices $A_{11} \in \End(\R^k)$, $A_{21} \in \LL(\R^k, \R^{n-k})$, $A_{22} \in \End(\R^{n-k})$ and $C_1 \in \LL(\R^k, \R^m)$.
For any solution of \eqref{E:Observer_system_const}, set similarly
\begin{equation*}
\begin{aligned}
\xhat =
\begin{pmatrix}
\xhat_{1} \\
\xhat_{2}
\end{pmatrix},\quad
\eps =
\begin{pmatrix}
\eps_{1} \\
\eps_{2}
\end{pmatrix},\quad
\bar{b} =
\begin{pmatrix}
b_{1} \\
b_{2}
\end{pmatrix},\quad
\Pric =
\begin{pmatrix}
\Pric_{11} & \Pric_{12}\\
\Pric^*_{12} & \Pric_{22}
\end{pmatrix}.
\end{aligned}
\end{equation*}

\begin{proposition}\label{prop:appli}

Assume that the pair $(C, A)$ is observable. Assume that for all  $T\in\GL(\R^n)$, for all $(\bar{f}, \bar{\LLL})$ as in \eqref{Hyp:T}, for all $k\in\{1,\dots,n\}$ and for all $(\bar{A}, \bar{C}) \in \End(\R^n) \times \LL(\R^n,\R^m)$ as in \eqref{eq:struct},
the following hypotheses hold.
\begin{enumerate}[label={\textnormal{H\arabic*.}}, ref={(\textnormal{H\arabic*})}]
    \item There exists $(f_{11}, \LLL_1)$ such that
    \begin{equation}\label{E:kalman_coupled_obs}
    \left\{
    \begin{aligned}
    &\dot{\xhat}_1= A_{11}\xhat_1 + b_1 - \LLL_1(\Pric_{11}, C_1)C_1 \eps_1
    \\
    &\dot{\eps}_1=\left(A_{11} - \LLL_1(\Pric_{11}, C_1)C_1\right) \eps_1
    \\
    &\dot{\Pric}_{11}= f_{11}(\Pric_{11}, A_{11}, C_1)
    \end{aligned}
    \right.
    \end{equation}
    where $(f_{11},\LLL_1)$ is such that
    \begin{align*}
    \bar{f}(\Pric,\bar{A}, \bar{C}) = 
    \begin{pmatrix}
    f_{11}(\Pric_{11}, A_{11}, C_1) && *\\
    * && *
    \end{pmatrix},\qquad
    \bar{\LLL}(\Pric, \bar{C}) = 
    \begin{pmatrix}
    \LLL_1(\Pric_{11}, C_1)\\
    *
    \end{pmatrix}.
    \end{align*}
    \label{Hyp:11}
    \item If $(C_1, A_{11})\in\LL(\R^k,\R^m)\times\End(\R^k)$ is an observable pair, then the solutions of \eqref{E:Observer_system_const} are such that for any initial conditions, ${\LLL}_{11}(\Pric_{11}(t), {C}_{1}){C}_{1}\eps_{1}(t)\to 0$ as $t\to+\infty$. \label{Hyp:conv}
\item For all $\Pric_{11}\in\sym_k$ and all ${C}_{1}\in\LL(\R^k,\R^m),\
\ker {\LLL}_{1}(\Pric_{11}, {C}_{1}) \cap \Image {C}_1 = \{0\}$. \label{Hyp:ker}
\end{enumerate}
Then the coupled system~\eqref{E:kalman_coupled} given by $(f(\cdot, \Au{u}, C),\LLL(\cdot, C))$ satisfies the hypothesis~\ref{NFOT}.
\end{proposition}

\begin{remark}
In the case where $T$ is the identity matrix and $k = n$, \ref{Hyp:11} is clearly satisfied, \ref{Hyp:conv} means that the correction term $\LLL(\Pric, \bar{C})\bar{C}\eps$ converges to zero for any observable pair $(\bar{A}, \bar{C})$, and \ref{Hyp:ker} means that the correction term is null if and only if $\bar{C}\eps = 0$. We will see in Theorem~\ref{thm:appli} that these hypotheses are clearly satisfied for the Luenberger and Kalman observers.
\end{remark}
\begin{remark}
Hypothesis~\ref{Hyp:11} can be seen as a compatibility condition between the observer dynamics and the Kalman observability decomposition: when $\bar{A}$ is of the standard form \eqref{eq:struct}, the observer acts autonomously on the upper left matrix block, which will correspond to the observable part of the system.
\end{remark}

This proposition is a consequence of the series of lemmas that follows.
Until the end of the proof of Proposition~\ref{prop:appli}, assume that its hypotheses are satisfied.
For any $\mu:\R^n\to \R$, $F_\mu$ denotes the vector field over $\R^n$ given by
\(
F_{\mu}(x)= \Au{\mu(x)} x+  b \mu(x).
\)

\begin{lemma}
\label{lem1}
For all $R>0$, there exists $\eta>0$ such that for all $\delta \in\VV_R$ satisfying $\sup\{\abs{\delta(x)}\mid x\in \Kx\} < \eta$, $0$ is the unique equilibrium point of $F_{\feed+\delta}$ lying in $\Kx$.
\end{lemma}

\begin{proof}
Let $R>0$ and $\delta\in\VV_R$. Let $x\in\Kx$ be such that $F_{\feed+\delta}(x) = 0$. Then,
\begin{align*}
0=F_{\feed+\delta}(x) = F_{\feed}(x) + \delta(x)(Bx + b).
\end{align*}
Then $|F_{\feed}(x)| = \abs{\delta(x)}|Bx + b|$.
Set $C_1 = \inf \{|F_{\feed}(x)|\mid x\in \Kx \backslash B(0, R)\} $. Since $0$ is not in the closure of $\Kx \backslash B(0, R)$, we get by uniqueness of the equilibrium point of $F_\feed$ that $C_1>0$.
Set also $C_2 = \sup \{|Bx+b|\mid x\in \Kx\}$. Since $\Kx$ is compact, $C_2<+\infty$.
Set $\eta = \frac{C_1}{C_2}$. Assume that $\sup \{\abs{\delta(x)}: x\in \Kx\} < \eta$. Then,
\begin{align*}
F_{\feed}(x) & \leq \eta \abs{Bx + b}
\leq C_1.
\end{align*}
Hence $x\in B(0, R)$ by definition of $C_1$. Then $\delta(x) = 0$. Hence $F_{\feed}(x) = 0$. Thus, $x=0$ since $0$ is the unique equilibrium point of $F_{\feed}$.
Moreover, by definition of $\VV_R$, $F_{\feed+\delta}(0) = 0$.
\end{proof}

\begin{lemma}
\label{lem2}
Assume that the pair $(C, A)$ is observable.
Let $(\cont_0, \xhat_0, \eps_0, \Pric_0) \in \R \times \R^n \times \R^n \times\sym$. Let $(\xhat, \eps, \Pric)$ be the solution of \eqref{E:Observer_system} given by the initial condition $(\xhat_0, \eps_0, \Pric_0)$ and the constant input $\cont\equiv\cont_0$.
If $\xhat$ is constant, then for all $t\in\R_+$, $\LLL(\Pric(t), C)C\eps(t) = 0$.
\end{lemma}

\begin{proof}
Let $(\cont_0, \xhat_0, \eps_0, \Pric_0) \in \R\times \R^n \times \R^n\times\sym $. Let $(\xhat, \eps, \Pric)$ be the solution of \eqref{E:Observer_system} given by the initial condition $(\xhat_0, \eps_0, \Pric_0)$ and the constant input $\cont\equiv\cont_0$.
Assume that $\xhat$ is constant, \ie $\xhat \equiv \xhat_0$.
Set $A_0=A+\cont_0 B$ and $b_0=b \cont_0 $. Then $\dot{\xhat} \equiv 0$ yields
\begin{align*}
A_0 \xhat + b_0 - \LLL(\Pric, C)C \eps \equiv 0.
\end{align*}
Since $\xhat$ is constant, so is $\LLL(\Pric)C \eps$. Then, set $K = \LLL(\Pric, C)C\eps$. It remains to show that $K=0$.

Let $k = \rank \OO(C, A_0)$ where $\OO(C, A_0)$ is defined by \eqref{E:Kalman-matrix}
Since $C\neq 0$ (since $(C, A_0)$ is observable), $k\geq1$.
According to the Kalman observability decomposition, there exists an invertible endomorphism $T\in\GL(\R^n)$ such that $\bar{A} = TA_0T^{-1}$ and $\bar{C} = CT^{-1}$ have the following structure:
\begin{equation}
\begin{aligned}
\bar{A} = 
\begin{pmatrix}
A_{11} & 0\\
A_{21} & A_{22}
\end{pmatrix},\qquad
\bar{C} = 
\begin{pmatrix}
C_1 & 0
\end{pmatrix},
\end{aligned}
\end{equation}
with suitable matrices $A_{11} \in \End\big(\R^k\big)$, $A_{21} \in \LL\big(\R^k, \R^{n-k}\big)$, $A_{22} \in \End\big(\R^{n-k}\big)$ and $C_1 \in \LL\big(\R^k, \R^m\big)$.
Moreover, the pair $(C_1, A_{11})$ is observable.
For the sake of readability, we omit the horizontal bars over the submatrices (for instance, $A_{11}$ is a submatrix of $\bar{A}$ and not of $A$).
Similarly, set
\begin{equation*}
\begin{array}{lllll}
\bar{x} = Tx =
\begin{pmatrix}
x_{1} \\
x_{2}
\end{pmatrix},\qquad &&
\bar{\xhat} = T\xhat =
\begin{pmatrix}
\xhat_{1} \\
\xhat_{2}
\end{pmatrix},\qquad &&
\bar{\eps} = T\eps =
\begin{pmatrix}
\eps_{1} \\
\eps_{2}
\end{pmatrix},\\
\bar{b}_0 = Tb_0 =
\begin{pmatrix}
b_{1} \\
b_{2}
\end{pmatrix},\qquad &&
\bar{K} = TK =
\begin{pmatrix}
K_{1} \\
K_{2}
\end{pmatrix},\qquad &&
\bar{\Pric} = T\Pric T^* =
\begin{pmatrix}
\Pric_{11} & \Pric_{12}\\
\Pric^*_{12} & \Pric_{22}
\end{pmatrix}.%
\end{array}
\end{equation*}
Then, according to \eqref{Hyp:T}, we have the following observed control system on $\bar{x}$, and the corresponding observer:

\begin{equation}\label{E:observation_system_normal}
\left\{
\begin{aligned}
&\dot{\bar{x}}= \bar{A} \bar{x}+ \bar{b}_0
\\
&y= \bar{C} \bar{x}\\
&\dot{\bar{\xhat}}=\bar{A} \bar{\xhat} + \bar{b}_0 - \bar{\LLL}(\Pric, \bar{C})\bar{C} \bar{\eps}
\\
&\dot{\bar{\eps}}=\left(\bar{A}- \bar{\LLL}(\Pric, \bar{C})\bar{C} \right) \bar{\eps}
\\
&\dot{\bar{\Pric}}= \bar{f}(\bar{\Pric}, \bar{A}, \bar{C}).
\end{aligned}
\right.
\end{equation}
Then, according to hypothesis~\ref{Hyp:11}, we can write
\begin{equation}
\left\{
\begin{aligned}
&\dot{\Pric}_{11}= f_{11}(\Pric_{11}, A_{11})
\\
&\dot{\xhat}_1= A_{11}\xhat_1 + b_1 - \LLL_1(\Pric_{11}, C_1)C_1 \eps_1
\\
&\dot{\eps}_1=\left(A_{11} - \LLL_1(\Pric_{11}, C_1)C_1\right) \eps_1.
\end{aligned}
\right.
\end{equation}
Since the pair $(C_1, A_{11})$ is observable,
\ref{Hyp:11} and \ref{Hyp:conv} yield
$\LLL_1(\Pric_{11}(t), C_1)C_1\eps_1(t) \to 0$ as $t\to +\infty$.
The equality $K_1 = \LLL_1(\Pric_{11}(t), C_1)C_1\eps_1(t)$
thus yields $K_1 = 0$.
Then, by hypotheses~\ref{Hyp:11} and \ref{Hyp:ker}, $\bar{C}\eps\equiv C_1\eps_1\equiv0$. Hence $K=0$.
Finally, we have $K = T^{-1}\bar{K} = 0$.
\end{proof}

\begin{lemma}
Let $(\delta, \xhat_0, \eps_0, \Pric_0) \in \A\times \K$. Let $(\xhat, \eps, \Pric)$ be the solution of \eqref{E:kalman_coupled} given by $(\delta, \xhat_0, \eps_0, \Pric_0)$.
Set $\cont_0 = (\feed+\delta)(\xhat_0)$.
Let $(\xhatan, \epsan, \Pan)$ be the solution of \eqref{E:Observer_system} given by the initial condition $(\xhat_0, \eps_0, \Pric_0)$ and the constant input $\cont\equiv\cont_0$.
If $\xhat^{(i)}(0) = 0$ for all $i\in\N\setminus\{0\}$, then $\xhatan$ is constant and
\begin{equation}
(\epsan^{(k)}(0), \Pan^{(k)}(0)) = (\eps^{(k)}(0), \Pric^{(k)}(0))
\end{equation}
for all $k\in\N$.
\label{lem4}
\end{lemma}
\begin{proof}
Assume that $\xhat^{(i)}(0) = 0$ for all $i\in\N\setminus\{0\}$.
Then, for all $i\in\N\setminus\{0\}$,
\begin{align}
\Au{(\feed+\delta)(\xhat)}^{(i)}(0)
= 0.
\label{eqA}
\end{align}
According to the ODE version of the Cauchy-Kovalevskaya theorem, $(\xhatan, \epsan, \Pan)$ is analytic in a neighborhood of $0$.
Hence,
it is sufficient to show that
\begin{equation}
(\xhatan^{(k)}(0), \epsan^{(k)}(0), \Pan^{(k)}(0)) = (\xhat^{(k)}(0), \eps^{(k)}(0), \Pric^{(k)}(0))
\label{eqan}
\end{equation}
for all $k\in\N$.
By definition of $(\xhat, \eps, \Pric)$ and $(\xhatan, \epsan, \Pan)$, we have 
\begin{align*}
(\xhatan(0), \epsan(0), \Pan(0)) = (\xhat_0, \eps_0, \Pric_0) = (\xhat(0), \eps(0), \Pric(0)).
\end{align*}
Let $k\in\N$. Assume that for all $i\in\intset{0}{k}$, \eqref{eqan} is satisfied. Then we prove that \eqref{eqan} is also satisfied for $i=k+1$. Using Faà di Bruno's formula and \eqref{eqA}, we get
\begin{align}
\Pric^{(k+1)}(0)
&= f\left( \Pric, \Au{(\feed+\delta)(\xhat)}, C \right)^{(k)}(0)\nonumber\\
&= f\left( \Pric, \Au{(\feed+\delta)(\xhat(0))}, C \right)^{(k)}(0)\tag{by \eqref{eqA}}\\
&= f\left( \Pan, \Au{(\feed+\delta)(\xhat(0))}, C \right)^{(k)}(0)\tag{by induction hypothesis}\\
&= \Pan^{(k+1)}(0).\nonumber
\end{align}
Likewise, we obtain
\(\eps^{(k+1)}(0) = \epsan^{(k+1)}(0)\) and \(\xhat^{(k+1)}(0) = \xhatan^{(k+1)}(0)\).
\end{proof}

\begin{lemma}\label{l:appli}
Assume that the pair $(C, A)$ is observable.
Let $(\xhat_0, \eps_0, \Pric_0) \in \K$.
Let $R>0$, $\eta>0$ as in Lemma \ref{lem1} and $\delta\in \VV_R$ satisfying $\sup\{\abs{\delta(x)} \mid x\in \Kx\} < \eta$. Let $(\xhat, \eps, \Pric)$ be the solution of \eqref{E:kalman_coupled} given by $(\delta, \xhat_0, \eps_0, \Pric_0)$.
If for all $i\in\N\setminus\{0\}$, $\xhat^{(i)}(0) = 0$,
then $\xhat \equiv \eps \equiv 0$.
\end{lemma}

\begin{proof}
Assume that for all $i\in\N\setminus\{0\}$, $\xhat^{(i)}(0) = 0$.
Set $\cont_0 = (\feed+\delta)(\xhat_0)$.
Let $(\xhatan, \epsan, \Pan)$ be the solution of \eqref{E:kalman_coupled} given by the initial condition $(\xhat_0, \eps_0, \Pric_0)$ and the constant input $\cont\equiv\cont_0$.
According to Lemma \ref{lem4}, $\xhatan \equiv \xhat_0$ and for all $k\in \N$, $(\epsan^{(k)}(0), \Pan^{(k)}(0)) = (\eps^{(k)}(0), \Pric^{(k)}(0))$.
Then,
by Lemma \ref{lem2}, we get that $\LLL(\Pan, C)C\epsan \equiv 0$.
Hence, $\Au{\cont_0}\xhatan + b \cont_0 \equiv 0$ \ie $\Au{(\feed+\delta)(\xhat_0)}\xhatan(t) + b (\feed+\delta)(\xhat_0) = 0$ for all $t\in\R_+$.
In particular, at $t=0$ we have that $F_{\feed+\delta}(\xhat_0) = 0$.
Hence, from Lemma \ref{lem1}, $\xhat_0 = 0$. By uniqueness of the solution of \eqref{E:kalman_coupled} for a given initial condition, it remains to prove that $\eps_0 = 0$ in order to get that $\xhat \equiv \eps \equiv 0$. Since the pair $(C, A)$ is observable, it is sufficient to prove that $CA^k\eps_0 = 0$ for all $k\in\N$. We proceed by induction.
From Lemma \ref{lem2}, $\LLL(\Pan(0), C)C\epsan(0) = 0$. Then, according to hypothesis~\ref{Hyp:ker}, $C\eps_0 = C\epsan(0) = 0$.
Let $k\in \N$.
Assume that $CA^i\eps_0 = 0$ for all $i\in\intset{0}{k-1}$. We prove in the following that $CA^k\eps_0 = 0$.
From Lemma \ref{lem2}, $(\LLL(\Pan, C)C\epsan)^{(i)}(0) = 0$ for all $i\in\N$. Hence, by Lemma \ref{lem4}, we get for all $i\in\N$,
$(\LLL(\Pric, C)C\eps)^{(i)}(0) = (\LLL(\Pan, C)C\epsan)^{(i)}(0) = 0$ and then
$C\eps^{(i)}(0) = C\Au{\cont_0}^i\eps_0 = CA^i\eps_0$ since $\cont_0 = (\feed+\delta)(\xhat_0) =(\feed+\delta)(0) = 0 $.
Then,
\begin{align}
0 &= (\LLL(\Pan, C)C\epsan)^{(k)}(0)
\tag{by Lemma \ref{lem2}}\\
&= (\LLL(\Pric, C)C\eps)^{(k)}(0)
\tag{by Lemma \ref{lem4}}\\
&= \sum_{i=0}^k \binom{k}{i} \LLL(\Pric, C)^{(k-i)}(0) C \eps^{(i)}(0)
\tag{by Leibniz rule}\\
&= \sum_{i=0}^k \binom{k}{i} \LLL(\Pric, C)^{(k-i)}(0) C A^i\eps_0
\nonumber\\
&=\LLL(\Pric_0, C)C A^k\eps_0.
\tag{by induction hypothesis}
\end{align}
Thus, by hypothesis~\ref{Hyp:ker}, $C A^k\eps_0 = 0$, which concludes the induction and the proof.
\end{proof}
This concludes the series of lemmas necessary to prove Proposition~\ref{prop:appli} and Theorem~\ref{thm:appli}.
\begin{proof}[Proof of Proposition~\ref{prop:appli}]
The statement follows directly from the contrapositive of Lemma~\ref{l:appli}.
\end{proof}
\begin{proof}[Proof of Theorem~\ref{thm:appli}]
Recall that, according to Proposition~\ref{prop:fc}, the Luenberger observer and the Kalman observer satisfy \ref{FC}. It remains to show that the sufficient conditions stated in the Proposition~\ref{prop:appli} are satisfied by these observers to conclude the proof of Theorem~\ref{thm:appli}.

Let $Q\in\sym_n$.
For all $(\bar{A}, \bar{C})\in\End(\R^n)\times\LL(\R^n,\R^m)$ and all $\Pric\in \sym_n$, let
\begin{align}
&f^\mathrm{Luenberger}(\Pric, \bar{A}, \bar{C}) = 0\tag{Luenberger observer}\\
&f^\mathrm{Kalman}_{Q}(\Pric, \bar{A}, \bar{C}) = \Pric \bar{A}^* + \bar{A} \Pric + Q - \Pric\bar{C}^*\bar{C}\Pric \tag{Kalman observer}
\end{align}
and $\LLL(\xi, C) = \Pric \bar{C}^*$.
Let $f\in\{f^\mathrm{Luenberger}, f^\mathrm{Kalman}_{Q}\}$.
According to Proposition \ref{prop:fc}, the time-varying vector field $f$ is forward complete.
For all $T\in\GL(\R^n)$, for all $(\bar{A}, \bar{C})\in\End(\R^n)\times\LL(\R^n,\R^m)$ and for all $\Pric\in \sym_n$, let $(\bar{f}, \bar{\LLL})$ be defined by
\begin{align}
\begin{cases}
    \bar{f}(T\Pric T^*, T\bar{A}T^{-1}, \bar{C}T^{-1}) = Tf(\Pric, \bar{A}, \bar{C})T^*\\
    \bar{\LLL}(T\Pric T^*, \bar{C}T^{-1}) = T\LLL(\Pric, \bar{C}).
\end{cases}
\end{align}
Then
\begin{align*}
\bar{\LLL}(T\Pric T^*, \bar{C}T^{-1})
= T\LLL(\Pric, \bar{C})
= T \Pric \bar{C}^*
= T \Pric T^* (\bar{C}  T^{-1})^*
= \LLL(T\Pric T^*, \bar{C}T^{-1}).
\end{align*}
Hence $\bar{\LLL} = \LLL$. Moreover,
if $f = f^\mathrm{Luenberger}$, then $\bar{f} = f = 0$.
Otherwise, if $f = f^\mathrm{Kalman}_{Q}$ and then
    \begin{align*}
    \bar{f}(T\Pric T^*, T\bar{A}T^{-1}, \bar{C}T^{-1})
    &= Tf(\Pric, \bar{A}, \bar{C})T^*\\
    &= T\Pric \bar{A}^* + \bar{A} \Pric + Q - \Pric\bar{C}^*\bar{C}\Pric T^*\\
    &= T\Pric T^* (T\bar{A}T^{-1})^* + (T\bar{A}T^{-1}) T\Pric T^* \\
    & \quad + TQT^* - T\Pric T^* (\bar{C}T^{-1})^*\bar{C}T^{-1} T\Pric T^*\\
    &= f^\mathrm{Kalman}_{TQT^*}(T\Pric T^*, T\bar{A}T^{-1}, \bar{C}T^{-1}),
    \end{align*}
Hence it is sufficient to prove that, for all $(\bar{A}, \bar{C}) \in \End(\R^n) \times \LL(\R^n,\R^m)$ satisfying \eqref{eq:struct}, $(f, \LLL)$ satisfies hypotheses~\ref{Hyp:11}, \ref{Hyp:conv} and \ref{Hyp:ker}.
Hypothesis~\ref{Hyp:11} requires some computations to check that if $(\bar{A}, \bar{C})$ is of the form \eqref{eq:struct}, then \eqref{E:kalman_coupled_obs} is satisfied with

\begin{align}
f_{11}(\Pric_{11}, \bar{A}_{11}, \bar{C}_{1}) =
\begin{cases}
0 & \text{ if } f = f^\mathrm{Luenberger}\\
\Pric_{11} \bar{A}_{11}^* + \bar{A}_{11} \Pric_{11} + Q_{11} - \Pric_{11}\bar{C}_{1}^*\bar{C}_{1}\Pric_{11} & \text{ if } f = f^\mathrm{Kalman}_{Q}
\end{cases}
\end{align}
and $\LLL_1(\Pric_{11}, \bar{C}_{1}) = \Pric_{11}\bar{C}_{1}^*$.
Hence, for any $f\in\{f^\mathrm{Luenberger}, f^\mathrm{Kalman}_{Q}\}$, $f_{11}$ is an observer of the same form than $f$ acting on $\R^k$.
Hypothesis~\ref{Hyp:conv} follows from the fact that these well-known observers guaranty that the correction term $\LLL_1(\Pric_{11}, \bar{C}_1)\bar{C}_1\eps_{1}$ goes to 0 as soon as the pair $(\bar{C}_1, \bar{A}_{11})$ is observable
(see \eg \cite[Chapter 1, Theorems 3 and 4]{Besancon}).
Hypothesis~\ref{Hyp:ker} is clear: for all $\Pric_{11}\in\sym_k$ and all $\bar{C}_1\in\LL(\R^k,\R^m)$, if $\eps_1\in\R^k$ is such that $\Pric_{11} \bar{C}_1^* \bar{C}_1 \eps_1 = 0$, then $\bar{C}_1\eps_1 = 0$ since $\Pric_{11}$ is invertible.
Thus the conclusion of Proposition~\ref{prop:appli} holds.
\end{proof}

\section*{Acknowledgments}
The authors would like to thank Vincent Andrieu and Daniele Astolfi for many fruitful discussions.
\startmodif
They would also like to thank the anonymous reviewer, for suggesting the addition of Corollary~\ref{cor:feedback} to the paper.
\stopmodif
\bibliographystyle{abbrv}
\bibliography{referencesv2}
\end{document}